\documentclass[11pt]{article}
\usepackage{amsfonts}
\usepackage{amsmath}
\usepackage{amssymb}
\usepackage{amsthm}
\usepackage{mathrsfs}
\usepackage{color}
\usepackage{hyperref}
\usepackage{a4}
\usepackage{fullpage}
\usepackage{url}
\usepackage{hyperref}
\hypersetup{
    unicode=true,          
    pdftoolbar=true,        
    pdfmenubar=true,        
    pdffitwindow=true,      
    pdftitle={Viscoelastic Paper},    
    pdfauthor={Stéphane Gerbi},     
    pdfcreator={TeXShop},   
    pdfproducer={TeXShop}, 
    pdfkeywords={}, 
    colorlinks=true,       
    linkcolor=blue,  
    citecolor=blue,        
    filecolor=magenta,      
    urlcolor=cyan           
 }

\newtheorem{theorem}{Theorem}[section]

\newtheorem{lemma}[theorem]{Lemma}

\newtheorem{definition}[theorem]{Definition}

\newtheorem{remark}[theorem]{Remark}

\newcommand{\R}{\mathbb{R}}

\begin{document}

\title{Global existence and exponential growth for a viscoelastic wave equation with dynamic
boundary conditions}
\author{St\'{e}phane Gerbi\thanks{%
Laboratoire de Math\'ematiques, Universit\'e de Savoie et CNRS, UMR-5128, 73376 Le Bourget du
Lac, France, E-mail : \url{stephane.gerbi@univ-savoie.fr}, } ~and Belkacem
Said-Houari\thanks{
Division of Mathematical and Computer Sciences and Engineering,  King Abdullah
University of Science and Technology (KAUST), Thuwal, KSA,
E-mail: \url{saidhouarib@yahoo.fr}}}
\date{}
\maketitle

\begin{abstract}
The goal of this work is to study a model of the wave equation with dynamic
boundary conditions and a viscoelastic term. First, applying the
Faedo-Galerkin method combined with the fixed point theorem, we show the
existence and uniqueness of a local in time solution. Second, we show that
under some restrictions on the initial data, the solution continues to exist
globally in time. On the other hand, if the interior source dominates the
boundary damping, then the solution is unbounded and grows as an exponential
function. In addition, in the absence of the strong damping, then the
solution ceases to exist and blows up in finite time.
\end{abstract}
{\bf Keywords:} Damped viscoelastic wave equations, global solutions, exponential growth, blow up in finite time, dynamic boundary conditions.

\section{Introduction}

We consider the following problem
\begin{equation}
\left\{
\begin{array}{ll}
u_{tt}-\Delta u-\alpha \Delta u_{t}+\displaystyle\int_{0}^{t}g(t-s)\Delta
u(s)ds=|u|^{p-2}u, & x\in \Omega ,\ t>0 \, ,\vspace{0.1cm} \\
u(x,t)=0, & x\in \Gamma _{0},\ t>0 \, ,\vspace{0.1cm} \\
u_{tt}(x,t)=-\left[ \displaystyle\frac{\partial u}{\partial \nu }(x,t)-%
\displaystyle\int_{0}^{t}g(t-s)\frac{\partial u}{\partial \nu }(x,s)ds+\frac{%
\alpha \partial u_{t}}{\partial \nu }(x,t)+h\left( u_{t}\right) \right] &
x\in \Gamma _{1},\ t>0\, , \vspace{0.1cm} \\
u(x,0)=u_{0}(x),\qquad u_{t}(x,0)=u_{1}(x) & x\in \Omega \, ,%
\end{array}%
\right.  \label{ondes}
\end{equation}%
%
%
%
%
where $u=u(x,t)\,,\,t\geq 0\,,\,x\in \Omega \,,\,\Delta $ denotes the
Laplacian operator with res\-pect to the $x$ variable, $\Omega $ is a
regular and bounded domain of $\mathbb{R}^{N}\,,\,(N\geq 1)$, $\partial
\Omega ~=~\Gamma _{0}~\cup ~\Gamma _{1}$, $mes(\Gamma _{0})>0,$ 
$\Gamma_{0}\cap \Gamma _{1}=\varnothing $ and $\partial/\partial \nu $ denotes the unit outer normal derivative, 
$\alpha$ is a positive constant, $p>2, \,h \mbox{ and } g$ are functions whose properties will be discussed in the next section,
$u_{0}\,,\,u_{1}$ are given functions.

Nowadays the wave equation with dynamic boundary conditions are used in a
wide field of applications. See \cite{MuerKu_2011} for some applications.
Problems similar to (\ref{ondes}) arise (for example) in the modeling of
longitudinal vibrations in a homogeneous bar in which there are viscous
effects. The term $\Delta u_t$, indicates that the stress is proportional
not only to the strain, but also to the strain rate, see \cite{CSh_76} fore more details. 

From the mathematical point of view, these problems do not neglect acceleration
terms on the boundary. Such type of boundary conditions are usually called
\textit{dynamic boundary conditions}. They are not only important from the
theoretical point of view but also arise in several physical applications.
For instance in one space dimension and for $g=0$, problem (\ref{ondes}) can
modelize the dynamic evolution of a viscoelastic rod that is fixed at one
end and has a tip mass attached to its free end. The dynamic boundary
conditions represents Newton's law for the attached mass, (see \cite%
{BST64,AKS96, CM98} for more details). In the two dimension space, as showed
in \cite{G06} and in the references therein, these boundary conditions arise
when we consider the transverse motion of a flexible membrane $\Omega $
whose boundary may be affected by the vibrations only in a region. Also some
dynamic boundary conditions as in problem (\ref{ondes}) appear when we
assume that $\Omega $ is an exterior domain of $\mathbb{R}^{3} $ in which
homogeneous fluid is at rest except for sound waves. Each point of the
boundary is subjected to small normal displacements into the obstacle: this type of dynamic boundary conditions are
known as acoustic boundary conditions,  see \cite{B76} for more details.

Littman and Markus \cite{LM88} considered a system which describe an elastic
beam, linked at its free end to a rigid body. The whole system is governed
by the Euler-Bernoulli Partial Differential Equations with dynamic boundary conditions. They used the
classical semigroup methods to establish existence and uniqueness results
while the asymptotic stabilization of the structure is achieved by the use
of feedback boundary damping.

In \cite{GV94} the author introduced the model%
\begin{equation}
u_{tt}-u_{xx}-u_{txx}=0,\qquad x\in (0,L),\,t>0,  \label{Vand_1}
\end{equation}%
which describes the damped longitudinal vibrations of a homogeneous flexible
horizontal rod of length $L$ when the end $x=0$ is rigidly fixed while the
other end $x=L$ is free to move with an attached load. Thus she considered
Dirichlet boundary condition at $x=0$ and dynamic boundary conditions at $%
x=L\,$, namely
\begin{equation}
u_{tt}(L,t)=-\left[ u_{x}+u_{tx}\right] (L,t),\qquad t>0  \ . \label{Vand_2}
\end{equation}%
By rewriting the whole system within the framework of the abstract theories
of the so-called $B$-evolution theory, the existence of a unique solution in
the strong sense has been shown. An exponential decay result was also proved
in \cite{GV96} for a problem related to (\ref{Vand_1})-(\ref{Vand_2}), which
describe the weakly damped vibrations of an extensible beam. See \cite{GV96}
for more details.

Subsequently, Zang and Hu \cite{ZH07}, considered the problem
\begin{equation*}
u_{tt}-p\left( u_{x}\right) _{xt}-q\left( u_{x}\right) _{x}=0,\qquad x\in
\left(0,1\right) ,\,t>0
\end{equation*}
with
\begin{equation*}
u\left( 0,t\right) =0,\qquad p\left( u_{x}\right) _{t}+q\left( u_{x}\right)
\left( 1,t\right) +ku_{tt}\left( 1,t\right) =0,\, t\geq 0.
\end{equation*}
By using the Nakao inequality, and under appropriate conditions on $p$ and $%
q $, they established both exponential and polynomial decay rates for the
energy depending on the form of the terms $p$ and $q$.

Recently, the present authors have
considered, in \cite{GS08} and \cite{GS082}, problem (\ref{ondes}) with $g=0$ and a nonlinear boundary 
damping of the form $h\left(u_{t}\right)=\left\vert u_{t}\right\vert^{m-2}u_{t}$. A local
existence result was obtained by combining the Faedo-Galerkin method with
the contraction mapping theorem. Concerning the asymptotic behavior, the
authors showed that the solution of such problem is unbounded and grows up
exponentially when time goes to infinity provided that the initial data are large
enough and the damping term is nonlinear. The blow up result was shown when
the damping is linear (i.e. $m=2$). Also, we proved in \cite{GS082} that
under some restrictions on the exponents $m$ and $p$, we can always find
initial data for which the solution is global in time and decays
exponentially to zero. These results had been recently generalized for a wide range of nonlinearities
in the equation and in the boundary term: the authors proved the local existence and uniqueness by a sophisticated
application of the non linear semigroup theory, see \cite{GRS2012}.

In the absence of the strong damping $\alpha \Delta u_t$ and for Dirichlet
boundary conditions on the whole boundary $\partial\Omega$, the question of
blow up in finite time of problem (\ref{ondes}) has been investigated by many authors. 
Messaoudi \cite{Mess03}  showed that if the initial energy is
negative and if the relaxation function $g$ satisfies the following assumption
\begin{equation}  \label{Messaoudi_condition}
\int_{0}^{\infty }g(s)ds<\frac{(p/2)-1}{(p/2)-1+(1/2p)} \ ,
\end{equation}%
then the solutions blow up in finite time.
In fact this last condition has been assumed by other researchers. See for instance \cite{Mess_Kafi_2007,Mess_Kafi_2008,MS2010,Mes01,SaZh2010,YWL2009}.

The main goal of this paper is to prove the local existence and to study the asymptotic
behavior of the solution of problem (\ref{ondes}).  

One of the main questions is to show a blow-up result of the solution. This question is a difficult open problem, since in the presence of the strong damping term, i.e. when $\alpha\neq 0$, the problem has a parabolic structure, which means that the solution gains more regularity. However, in this paper, we give a partial answer to this question and show that for $\alpha\neq 0$ and for large initial data, the solution is unbounded and grows exponentially as $t$ goes to infinity.
While for the case $\alpha=0$, the solution has been shown to blow up in finite time.

The main contribution of this paper in this blow up result is the following:
the exponential growth and blow-up results  hold without making the assumption (\ref{Messaoudi_condition}).
In fact the only requirement is that the exponent $p$ has to be large enough which is a condition 
much weaker than condition (\ref{Messaoudi_condition}). Moreover, unlike in the works of Messaoudi and coworkers,
we do not assume any polynomial structure on the damping term $h(u_t)$,
to obtain an exponential growth of the solution or a blow up in finite time.

This paper is organized as follows:
firstly, applying the Faedo-Galerkin method combined with the fixed point theorem, we show,  in Section \ref{local_existence_section}, the
existence and uniqueness of a local in time solution. 
Secondly, under the smallness assumption on the initial
data, we show, in Section \ref{Global_existence_section}, that the solution continues to exist globally in time. 
On the other hand, in Section \ref{Exponential_growth_section}, we prove that under
some restrictions on the initial data and if the interior source dominates
the boundary damping then the $L^p$-norm of the solution grows as an
exponential function. Lastly, in Section \ref{blow_up_section}, we investigate the case when $\alpha=0$ and we prove that the
solution ceases to exist and blows up in finite time.

\section{Preliminary and local existence}

\label{local_existence_section}

In this section, we introduce some notations used throughout this paper.
We also prove a local existence result of the solution of problem (\ref%
{ondes}).

We denote
\begin{equation*}
H_{\Gamma_{0}}^{1}(\Omega) =\left\{u \in H^1(\Omega) /\ u_{\Gamma_{0}} =
0\right\} .
\end{equation*}
By $( .,.) $ we denote the scalar product in $L^{2}( \Omega)$ i.e. $(u,v)(t)
= \int_{\Omega} u(x,t) v(x,t) dx$. Also we mean by $\Vert
.\Vert_{q}$ the $L^{q}(\Omega) $ norm for $1 \leq q \leq \infty$, and by $%
\Vert .\Vert_{q,\Gamma_{1}}$ the $L^{q}(\Gamma_{1}) $ norm.

Let $T>0$ be a real number and $X$ a Banach space endowed with norm $\Vert.\Vert _{X}$. 
$L^{p}(0,T;X),\ 1~\leq p~<\infty $ denotes the space of
functions $f$ which are $L^{p}$ over $\left( 0,T\right) $ with values in $X$%
, which are measurable and $\Vert f\Vert _{X}\in L^{p}\left( 0,T\right) $.
This space is a Banach space endowed with the norm
\begin{equation*}
\Vert f\Vert _{L^{p}\left( 0,T;X\right) }=\left( \int_{0}^{T}\Vert f\Vert
_{X}^{p}dt\right) ^{1/p}\quad .
\end{equation*}%
$L^{\infty }\left( 0,T;X\right) $ denotes the space of functions $f:\left]
0,T\right[ \rightarrow X$ which are measurable and $\Vert f\Vert _{X}\in
L^{\infty }\left( 0,T\right) $. This space is a Banach space endowed with
the norm:
\begin{equation*}
\Vert f\Vert _{L^{\infty }(0,T;X)}=\mbox{ess}\sup_{0<t<T}\Vert f\Vert
_{X}\quad .
\end{equation*}%
We recall that if $X$ and $Y$ are two Banach spaces such that $%
X\hookrightarrow Y$ (continuous embedding), then
\begin{equation*}
L^{p}\left( 0,T;X\right) \hookrightarrow L^{p}\left( 0,T;Y\right) ,\ 1\leq
p\leq \infty .
\end{equation*}%
We will also use the embedding (see \cite[Therorem 5.8]{A75}):
\begin{equation*}
H_{\Gamma _{0}}^{1}(\Omega )\hookrightarrow L^{p}(\Omega),\;2\leq p\leq
\bar{p}\quad \mbox{where }\quad \bar{p}=\left\{
\begin{array}{ll}
\dfrac{2N}{N-2} & \mbox{ if }N\geq 3, \\
+\infty & \mbox{ if }N=1,2 \ ,%
\end{array}%
\right.
\end{equation*}
and also
\begin{equation*}
H_{\Gamma _{0}}^{1}(\Omega )\hookrightarrow L^{q}(\Gamma _{1}),\;2\leq q\leq
\bar{q}\quad \mbox{where }\quad \bar{q}=\left\{
\begin{array}{ll}
\dfrac{2(N-1)}{N-2} & \mbox{ if }N\geq 3, \\
+\infty & \mbox{ if }N=1,2.%
\end{array}%
\right.
\end{equation*}%
For $2 \leq m \leq \bar{q}$, let us denote $V=H_{\Gamma _{0}}^{1}(\Omega )\cap L^{m}(\Gamma _{1})$.

We assume that the relaxation functions $g$ is of class $C^{1}$ on $\R$ and
satisfies:
\begin{equation}
\forall \, s \, \in \R \,,\, g\left( s\right) \geq 0,\mbox{ and } \left(1-\displaystyle\int_{0}^{\infty }g\left(s\right) ds \right)=l>0   \ .\label{hypothesis_g}
\end{equation}%
Moreover, we suppose that:
\begin{equation}
\forall \, s\geq 0 \,,\, g^{\prime}(s) \leq 0.
\label{hypothesis_g_2}
\end{equation}
The hypotheses on the function $h$ are the following:
\begin{description}
\item[(H1)] $h$ is continuous and strongly monotone, i.e.  for $2 \leq m \leq \bar{q}$, there exists a
constant $m_{0}>0$ such that
\begin{equation}
\left( h(s)-h(v)\right) (s-v)\geq m_{0}|s-v|^{m} \, , 
\label{Assumption_h_1}
\end{equation}

\item[(H2)] there exist two positive constants $c_{m}$ and $C_{m}$ such that
\begin{equation}
c_{m}|s|^{m}\leq h(s)s\leq C_{m}|s|^{m},\qquad \forall s\in \mathbb{R} \ .
\label{Assumption_h}
\end{equation}
\end{description}

For a function $u \in C\Bigl(\lbrack \ 0,T],H_{\Gamma _{0}}^{1}(\Omega )\Bigl)$, let us introduce the following notation:
\begin{equation*}
\left( g\diamond u\right) \left( t\right) =\int_{0}^{t}g\left( t-s\right)
\left\Vert \nabla u\left( s\right) -\nabla u\left( t\right) \right\Vert
_{2}^{2}ds.
\end{equation*}%
Thus, when $u \in C\Bigl(\lbrack \ 0,T],H_{\Gamma _{0}}^{1}(\Omega )\Bigl)\cap C^{1}\Bigl(\lbrack \ 0,T],L^{2}(\Omega )\Bigl)$ such that
$u_{t} \in L^{2}\Bigl(0,T;H_{\Gamma _{0}}^{1}(\Omega )\Bigl)$, we have:
\begin{eqnarray}
\frac{d}{dt}\left( g\diamond u\right) \left( t\right)
&=&\int_{0}^{t}g^{\prime }\left( t-s\right) \left\Vert \nabla u\left(
s\right) -\nabla u\left( t\right) \right\Vert _{2}^{2}ds  \notag \\
&&+\frac{d}{dt}\left( \left\Vert \nabla u\left( t\right) \right\Vert
_{2}^{2}\right) \int_{0}^{t}g\left( s\right) ds-2\int_{\Omega
}\int_{0}^{t}g\left( t-s\right) \nabla u\left( s\right) \nabla u_{t}\left(
t\right) dsdx  \notag \\
&=&\left( g^{\prime }\diamond u\right) \left( t\right) -2\int_{\Omega
}\int_{0}^{t}g\left( t-s\right) \nabla u\left( s\right) \nabla u_{t}\left(
t\right) dsdx  \label{Integral_relation} \\
&&+\frac{d}{dt}\left\{ \left\Vert \nabla u\left( t\right) \right\Vert
_{2}^{2}\int_{0}^{t}g\left( s\right) ds\right\} -g\left( t\right) \left\Vert
\nabla u\left( t\right) \right\Vert _{2}^{2}.  \notag
\end{eqnarray}%
This last identity implies:
\begin{eqnarray}
\int_{\Omega }\int_{0}^{t}g\left( t-s\right) \nabla u\left( s\right) \nabla
u_{t}\left( t\right) dsdx &=&\frac{1}{2}\left( g^{\prime }\diamond u\right)
\left( t\right) +\frac{1}{2}\frac{d}{dt}\left\{ \left\Vert \nabla u\left(
t\right) \right\Vert _{2}^{2}\int_{0}^{t}g\left( s\right) ds\right\}  \notag
\\
&&-\frac{1}{2}g\left( t\right) \left\Vert \nabla u\left( t\right)
\right\Vert _{2}^{2}-\frac{1}{2}\frac{d}{dt}\left( g\diamond u\right) \left(
t\right) .  \label{Nonlinear_term_viscoelastic}
\end{eqnarray}%
For $u \in C\Bigl(\lbrack \ 0,T],H_{\Gamma _{0}}^{1}(\Omega )\Bigl)\cap C^{1}\Bigl(\lbrack \ 0,T],L^{2}(\Omega )\Bigl)$ such that
$u_{t} \in L^{2}\Bigl(0,T;H_{\Gamma _{0}}^{1}(\Omega )\Bigl)$, let us define the modified energy functional $E$ by:
\begin{eqnarray}
E\left( t,u,u_{t}\right) =E\left( t\right) &=&\dfrac{1}{2}\Vert u_{t}\left(t\right) \Vert _{2}^{2}+\dfrac{1}{2}\Vert u_{t}\left( t\right) \Vert
_{2,\Gamma _{1}}^{2}+\dfrac{1}{2}\left( 1-\displaystyle\int_{0}^{t}g\left(
s\right) ds\right) \left\Vert \nabla u\left( t\right) \right\Vert _{2}^{2}\vspace{0.2cm} \notag\\
& ~&+\dfrac{1}{2}\left( g\diamond u\right) \left( t\right) -\dfrac{1}{p}\left\Vert u\left( t\right) \right\Vert _{p}^{p}.\label{Energy_visco_elastic}
\end{eqnarray}

The following local existence result of the solution of problem (\ref{ondes}) is closely related
to the one we have proved for a slightly different problem in \cite[Theorem 2.1]{GS08},
where no memory term was present. Let us sate it:
\begin{theorem}
\label{existence} Assume that (\ref{hypothesis_g}), (\ref{hypothesis_g_2})
and (\ref{Assumption_h_1}) hold.
Let $2\leq p\leq \bar{q}$ and $\max\left( 2,\frac{\bar{q}}{\bar{q}+1-p} \right) \leq m \leq \bar{q}$. 
 Then given $%
u_{0}\in H_{\Gamma _{0}}^{1}(\Omega )$ and $u_{1}\in L^{2}(\Omega )$, there
exists $T>0$ and a unique solution $u$ of the problem (\ref{ondes}) on $%
(0,T) $ such that
\begin{eqnarray*}
u &\in &C\Bigl(\lbrack \ 0,T],H_{\Gamma _{0}}^{1}(\Omega )\Bigl)\cap C^{1}%
\Bigl(\lbrack \ 0,T],L^{2}(\Omega )\Bigl), \\
u_{t} &\in &L^{2}\Bigl(0,T;H_{\Gamma _{0}}^{1}(\Omega )\Bigl)\cap
L^{m}\left( \left( 0,T\right) \times \Gamma _{1}\right) .
\end{eqnarray*}
\end{theorem}

Let us mention that Theorem \ref{existence} also holds for $\alpha=0$. The proof of Theorem \ref{existence} can be done along the same line as in
\cite[Theorem 2.1]{GS08}. The main idea of the proof is based on the
combination between the Fadeo-Galerkin approximations and the contraction
mapping theorem. However, for the convenience of the reader we give only the
outline of the proof here.

For $u\in C\bigl(\lbrack 0,T],H_{\Gamma _{0}}^{1}(\Omega )\bigl)\,\cap
\,C^{1}\bigl(\lbrack 0,T],L^{2}(\Omega )\bigl)$ given, let us consider the
following problem:
\begin{equation}
\left\{
\begin{array}{ll}
v_{tt}-\Delta v-\alpha \Delta v_{t}+\displaystyle\int_{0}^{t}g(t-s)\Delta
v(s)ds=|u|^{p-2}u, & x\in \Omega ,\ t>0 \,,\vspace{0.2cm} \\
v(x,t)=0, & x\in \Gamma _{0},\ t>0\,, \vspace{0.2cm} \\
v_{tt}(x,t)=-\left[ \displaystyle\frac{\partial v}{\partial \nu }(x,t)-%
\displaystyle\int_{0}^{t}g(t-s)\frac{\partial v}{\partial \nu }(x,s)ds+\frac{%
\alpha \partial v_{t}}{\partial \nu }(x,t)+h\left( v_{t}\right) \right] &
x\in \Gamma _{1},\ t>0 \,, \vspace{0.2cm} \\
v(x,0)=u_{0}(x),\;v_{t}(x,0)=u_{1}(x) & x\in \Omega .%
\end{array}%
\right.  \label{ondes_u}
\end{equation}

\begin{definition}
\label{generalised} A function $v(x,t)$ such that
\begin{eqnarray*}
v &\in &L^{\infty }\left( 0,T;H_{\Gamma _{0}}^{1}(\Omega )\right) \ , \\
v_{t} &\in &L^{2}\left( 0,T;H_{\Gamma _{0}}^{1}(\Omega )\right) \cap
L^{m}\left( (0,T)\times \Gamma _{1}\right) \ , \\
v_{t} &\in &L^{\infty }\left( 0,T;H_{\Gamma _{0}}^{1}(\Omega )\right) \cap
L^{\infty }\left( 0,T;L^{2}(\Gamma _{1})\right) \ , \\
v_{tt} &\in &L^{\infty }\left( 0,T;L^{2}(\Omega )\right) \cap L^{\infty
}\left( 0,T;L^{2}(\Gamma _{1})\right) \ , \\
v(x,0) &=&u_{0}(x)\,, \\
v_{t}(x,0) &=&u_{1}(x)\,,
\end{eqnarray*}%
is a generalized solution to the problem (\ref{ondes_u}) if for any function
$\omega \in H_{\Gamma _{0}}^{1}(\Omega )\cap L^{m}(\Gamma _{1})$ and $%
\varphi \in C^{1}(0,T)$ with $\varphi (T)=0$, we have the following
identity:
\begin{equation*}
\begin{array}{lll}
\displaystyle\int_{0}^{T}(|u|^{p-2}u,w)(t)\,\varphi (t)\,dt & =\displaystyle%
\int_{0}^{T}\Bigg[(v_{tt},w)(t)+(\nabla v,\nabla
w)(t)-\int_{0}^{t}g(t-s)(\nabla v\left( s\right) ,\nabla w\left( t\right) )ds %
 &  \\
& +\alpha (\nabla v_{t},\nabla w)(t)\varphi (t)dt\Bigg]+\displaystyle%
\int_{0}^{T}\varphi (t)\left( \int_{\Gamma _{1}}v_{tt}(t)w\,d\Gamma +h\left(
v_{t}\right) \,d\Gamma \right) dt. &
\end{array}%
\end{equation*}
\end{definition}

\begin{lemma}
\label{existence_f} Let $2\leq p\leq \bar{q}$ and $2 \leq m \leq \bar{q}$. Let $u_{0}\in H^{2}(\Omega
)\cap V,\,u_{1}\in H^{2}(\Omega )$, then for any $T>0,$ there exists a
unique generalized solution (in the sense of Definition \ref{generalised}), $%
v(t,x)$ of problem (\ref{ondes_u}).
\end{lemma}

The proof of Lemma \ref{existence_f} is essentially based on the
Fadeo-Galerkin approximations combined with the compactness method and can
be done along the same line as in \cite[Lemma 2.2]{GS08}, we omit the details.

In the following lemma we state a local existence result of problem (\ref%
{ondes_u}).

\begin{lemma}
\label{existence_u} Let $2\leq p\leq \bar{q}$ and $\max\left( 2, \frac{\bar{q}}{\bar{q}+1-p} \right) \leq m \leq \bar{q}$. Then given $u_{0}~\in
~H_{\Gamma _{0}}^{1}(\Omega )\,,\,u_{1}\in L^{2}(\Omega )$ there exists $T>0$
and a unique solution $v$ of the problem (\ref{ondes_u}) on $(0,T)$ such
that
\begin{eqnarray*}
v &\in &C\Bigl(\lbrack 0,T],H_{\Gamma _{0}}^{1}(\Omega )\Bigl)\,\cap \,C^{1}%
\Bigl(\lbrack 0,T],L^{2}(\Omega )\Bigl), \\
v_{t} &\in &L^{2}\Bigl(0,T;H_{\Gamma _{0}}^{1}(\Omega )\Bigl)\,\cap
L^{m}\left( \left( 0,T\right) \times \Gamma _{1}\right)
\end{eqnarray*}%
and satisfies the energy inequality:
\begin{eqnarray*}
&&\frac{1}{2}\left[ \Vert u_{t}\left( t\right) \Vert _{2}^{2}+\Vert
u_{t}\left( t\right) \Vert _{2,\Gamma _{1}}^{2}+\left( 1-\displaystyle%
\int_{0}^{t}g\left( s\right) ds\right) \left\Vert \nabla u\left( t\right)
\right\Vert _{2}^{2}\vspace{0.2cm}+\left( g\diamond u\right) \left( t\right) %
\right] _{s}^{t} \\
&&+\alpha \displaystyle\int_{s}^{t}\Vert \nabla v_{t}(\tau )\Vert
_{2}^{2}d\tau +\displaystyle\int_{s}^{t}\int_{\Gamma _{1}}h(v_{t}(\sigma,\tau
))d\sigma d\tau \\
&\leq &\displaystyle\int_{s}^{t}\displaystyle\int_{\Omega }|u(\tau
)|^{p-2}u(\tau )v_{t}(\tau )d\tau dx
\end{eqnarray*}%
for $0\leq s\leq t\leq T$.
\end{lemma}

\begin{proof} We first approximate $u\in C([0,T],H_{\Gamma
_{0}}^{1}(\Omega ))\cap C^{1}\left( [0,T],L^{2}(\Omega )\right) $ endowed
with the standard norm $\Vert u\Vert =\displaystyle\max_{t\in \lbrack
0,T]}\Vert u_{t}(t)\Vert _{2}+\Vert u(t)\Vert _{H^{1}(\Omega )}$, by a
sequence $(u^{k})_{k\in \mathbb{N}}\subset C^{\infty }([0,T]\times \overline{%
\Omega })$ by a standard convolution arguments (see \cite{B83}). Next, we
approximate the initial data $u_{1}\in L^{2}(\Omega )$ by a sequence $%
(u_{1}^{k})_{k\in \mathbb{N}}\subset C_{0}^{\infty }(\Omega )$. Finally,
since the space $H^{2}(\Omega )\cap V\cap H_{\Gamma _{0}}^{1}(\Omega )$ is
dense in $H_{\Gamma _{0}}^{1}(\Omega )$ for the $H^{1}$ norm, we approximate
$u_{0}\in H_{\Gamma _{0}}^{1}(\Omega )$ by a sequence $(u_{0}^{k})_{k\in
\mathbb{N}}\subset H^{2}(\Omega )\cap V\cap H_{\Gamma _{0}}^{1}(\Omega )$.

We consider now the set of the following problems:
\begin{equation}
\left\{
\begin{array}{ll}
v_{tt}^{k}-\Delta v^{k}-\alpha \Delta v_{t}^{k}+\displaystyle%
\int_{0}^{t}g(t-s)\Delta v^{k}(s)ds=|u^{k}|^{p-2}u^{k}, & x\in \Omega ,\ t>0%
\vspace{0.2cm}, \\
v^{k}(x,t)=0, & x\in \Gamma _{0},\ t>0\vspace{0.2cm}, \\
v_{tt}^{k}(x,t)=-\left[ \displaystyle\frac{\partial v^{k}}{\partial \nu }%
(x,t)-\displaystyle\int_{0}^{t}g(t-s)\frac{\partial v^{k}}{\partial \nu }%
(x,s)ds+\frac{\alpha \partial v_{t}^{k}}{\partial \nu }(x,t)+h\left(
v_{t}^{k}\right) \right] , & x\in \Gamma _{1},\ t>0\vspace{0.2cm}, \\
v^{k}(x,0)=u_{0}^{k},\;v_{t}^{k}(x,0)=u_{1}^{k}, & x\in \Omega .%
\end{array}%
\right.  \label{approx_k}
\end{equation}%
Since every hypothesis of Lemma \ref{existence_f} are verified, we can find
a sequence of unique solution $\left( v_{k}\right) _{k\in \mathbb{N}}$ of
the problem (\ref{approx_k}). Our goal now is to show that $%
(v^{k},v_{t}^{k})_{k\in \mathbb{N}}$ is a Cauchy sequence in the space
\begin{equation*}
\begin{array}{lll}
Y_{T} & =\Bigl\{ & (v,v_{t})/v\in C\left( \left[ 0,T\right] ,H_{\Gamma
_{0}}^{1}(\Omega )\right) \cap C^{1}\left( \left[ 0,T\right] ,L^{2}(\Omega
)\right) , \\
&  & v_{t}\in L^{2}\left( 0,T;H_{\Gamma _{0}}^{1}(\Omega )\right) \cap
L^{m}\left( \left( 0,T\right) \times \Gamma _{1}\right) \Bigl\}%
\end{array}%
\end{equation*}%
endowed with the norm
\begin{equation} \label{norm}
\Vert (v,v_{t})\Vert _{Y_{T}}^{2}=\max_{0\leq t\leq T}\Bigl[\Vert v_{t}\Vert
_{2}^{2}+l\Vert \nabla v\Vert _{2}^{2}\Bigl]+\Vert v_{t}\Vert _{L^{m}\left(
\left( 0,T\right) \times \Gamma _{1}\right) }^{2}+\int_{0}^{t}\Vert \nabla
v_{t}(s)\Vert _{2}^{2}\;ds\;.
\end{equation}%
For this purpose, we set $U=u^{k}-u^{k^{\prime }},\ V=v^{k}-v^{k^{\prime }}$%
. It is straightforward to see that $V$ satisfies:
\begin{equation*}
\left\{
\begin{array}{ll}
V_{tt}-\Delta V-\alpha \Delta V_{t}+\displaystyle\int_{0}^{t}g(t-s)\Delta
V(s)ds=|u^{k}|^{p-2}u^{k}-|u^{k^{\prime }}|^{p-2}u^{k^{\prime }} & x\in
\Omega ,\ t>0 \,, \vspace{0.2cm} \\
V(x,t)=0 & x\in \Gamma _{0},\ t>0 \,, \vspace{0.2cm} \\
V_{tt}(x,t)=-\left[\displaystyle\frac{\partial V}{\partial \nu }(x,t)-%
\displaystyle\int_{0}^{t}g(t-s)\frac{\partial V}{\partial \nu }(x,s)ds+\frac{%
\alpha \partial V_{t}}{\partial \nu }(x,t)+h(v_{t}^{k})-h(v_{t}^{k^{\prime
}})\right] & x\in \Gamma _{1},\ t>0 \,, \vspace{0.2cm} \\
V(x,0)=u_{0}^{k}-u_{0}^{k^{\prime }},\;V_{t}(x,0)=u_{1}^{k}-u_{1}^{k^{\prime
}} & x\in \Omega .%
\end{array}%
\right.
\end{equation*}%
We multiply the above differential equations by $V_{t}$, we integrate over $%
(0,t)\times \Omega $, we use integration by parts and the identity (\ref%
{Integral_relation}) to obtain:
\begin{equation}
\left.
\begin{array}{ll}
&
\begin{array}{l}
\dfrac{1}{2}\left( \Vert V_{t}\left( t\right) \Vert _{2}^{2}+\Vert
V_{t}\left( t\right) \Vert _{2,\Gamma _{1}}^{2}+\left( 1-\displaystyle%
\int_{0}^{t}g\left( r\right) dr\right) \left\Vert \nabla V\left( t\right)
\right\Vert _{2}^{2}\right) \\
+\alpha \displaystyle\int_{0}^{t}\Vert \nabla V_{t}\Vert _{2}^{2}ds\vspace{%
0.2cm}+\int_{0}^{t}\int_{\Gamma _{1}}\left( h(v_{t}^{k}(x,\tau
))-h(v_{t}^{k^{\prime }}(x,\tau ))\right) \left( v_{t}^{k}(x,\tau
)-v_{t}^{k^{\prime }}(x,\tau )\right) d\Gamma d\tau%
\end{array}
\\
& -\dfrac{1}{2}\displaystyle\int_{0}^{t}\left( g^{\prime }\diamond V\right)
\left( s\right) ds+\dfrac{1}{2}\displaystyle\int_{0}^{t}g\left( s\right)
\left\Vert \nabla V\left( s\right) \right\Vert _{2}^{2}ds\vspace{0.2cm} \\
= & \displaystyle\frac{1}{2}\left( \Vert V_{t}(0)\Vert _{2}^{2}+\Vert \nabla
V(0)\Vert _{2}^{2}+\Vert V_{t}(0)\Vert _{2,\Gamma _{1}}^{2}\right) \vspace{%
0.2cm} \\
& +\displaystyle\int_{0}^{t}\displaystyle\int\limits_{\Omega }\left(
|u^{k}|^{p-2}u^{k}-|u^{k^{\prime }}|^{p-2}u^{k^{\prime }}\right) \left(
v_{t}^{k}-v_{t}^{k^{\prime }}\right) \,dxds,\quad \forall t\in \left(
0,T\right) .%
\end{array}%
\right.  \label{Main_estimate_existence}
\end{equation}%
Consequently, the above inequality together with (\ref{hypothesis_g}), (\ref%
{hypothesis_g_2}) and (\ref{Assumption_h_1}) gives%
\begin{equation}
\left.
\begin{array}{ll}
& \dfrac{1}{2}\left( \Vert V_{t}\left( t\right) \Vert _{2}^{2}+\Vert
V_{t}\left( t\right) \Vert _{2,\Gamma _{1}}^{2}+l\left\Vert \nabla V\left(
t\right) \right\Vert _{2}^{2}\right) +\alpha \displaystyle\int_{0}^{t}\Vert
\nabla V_{t}\Vert _{2}^{2}ds\vspace{0.2cm}+m_{0}\int_{0}^{t}\Vert V_{t}\Vert
_{m,\Gamma _{1}}^{m}ds \\
\leq & \displaystyle\frac{1}{2}\left( \Vert V_{t}(0)\Vert _{2}^{2}+\Vert
\nabla V(0)\Vert _{2}^{2}+\Vert V_{t}(0)\Vert _{2,\Gamma _{1}}^{2}\right) \\
& \vspace{0.2cm}+\displaystyle\int_{0}^{t}\displaystyle\int\limits_{\Omega
}\left( |u^{k}|^{p-2}u^{k}-|u^{k^{\prime }}|^{p-2}u^{k^{\prime }}\right)
\left( v_{t}^{k}-v_{t}^{k^{\prime }}\right) \,dxds,\quad \forall t\in \left(
0,T\right) .%
\end{array}%
\right.  \label{Main_inequality_Cauchy}
\end{equation}%
Following the same method as in \cite{GS08}, we deduce that there exists $C$
depending only on $\Omega \mbox{ and }p$ such that:
\begin{equation*}
\Vert V\Vert _{Y_{T}}\leq C\left( \Vert V_{t}(0)\Vert _{2}^{2}+\Vert \nabla
V(0)\Vert _{2}^{2}+\Vert V_{t}(0)\Vert _{2,\Gamma _{1}}^{2}\right) +CT\Vert
U\Vert _{Y_{T}}.
\end{equation*}%
Since $(u_{0}^{k})_{k\in \mathbb{N}}$ is a converging sequence in $H_{\Gamma
_{0}}^{1}\left( \Omega \right) $, $(u_{1}^{k})_{k\in \mathbb{N}}$ is a
converging sequence in $L^{2}\left( \Omega \right) $ and $\left(
u^{k}\right) _{k\in \mathbb{N}}$ is a converging sequence in $C\left( \left[
0,T\right] ,H_{\Gamma _{0}}^{1}(\Omega )\right) \cap C^{1}\left( \left[ 0,T%
\right] ,L^{2}(\Omega )\right) $ (so in $Y_{T}$ also), we conclude that $%
(v^{k},v_{t}^{k})_{k\in \mathbb{N}}$ is a Cauchy sequence in $Y_{T}$. Thus $%
(v^{k},v_{t}^{k})$ converges to a limit $(v,v_{t})\in Y_{T}$. Now by the
same procedure used by Georgiev and Todorova in \cite{GT94}, we prove that
this limit is a weak solution of the problem (\ref{ondes_u}). This completes
the proof of the Lemma \ref{existence_u}.
\end{proof}
\begin{proof}[Proof of Theorem \ref{existence}]
In order to prove Theorem \ref{existence}, we use the contraction mapping
theorem.\newline
Indeed, for $T>0,$ let us define the convex closed subset of $Y_{T}$:
\begin{equation*}
X_{T}=\left\{ (v,v_{t})\in Y_{T}\mbox{ such that }v(0)=u_{0},v_{t}(0)=u_{1}%
\right\} .
\end{equation*}%
Let us denote:
\begin{equation*}
B_{R}\left( X_{T}\right) =\left\{ v\in X_{T};\Vert v\Vert _{Y_{T}}\leq
R\right\} ,
\end{equation*}%
the ball of radius $R$ in $X_{T}$. Then, Lemma \ref{existence_u} implies
that for any $u\in X_{T}$, we may define $v=\Phi \left( u\right) $ the
unique solution of (\ref{ondes_u}) corresponding to $u$. Our goal now is to
show that for a suitable $T>0$, $\Phi $ is a contractive map satisfying $%
\Phi \left( B_{R}(X_{T})\right) \subset B_{R}(X_{T})$. \newline
Let $u\in B_{R}(X_{T})$ and $v=\Phi \left( u\right) $. Then for all $t\in
\lbrack 0,T]$ we have as in (\ref{Main_estimate_existence}):
\begin{equation}
\begin{array}{l}
\Vert v_{t}\Vert _{2}^{2}+l\Vert \nabla v\Vert _{2}^{2}+\Vert v_{t}\Vert
_{2,\Gamma _{1}}^{2}+2\alpha \displaystyle\int\limits_{0}^{t}\Vert \nabla
v_{t}\Vert _{2}^{2}\,ds+c\int_{0}^{t}\Vert v_{t}\Vert _{m,\Gamma _{1}}^{m}ds
\\
\leq \Vert u_{1}\Vert _{2}^{2}+\Vert \nabla u_{0}\Vert _{2}^{2}+\Vert
u_{1}\Vert _{2,\Gamma _{1}}^{2}+2\displaystyle\int\limits_{0}^{t}%
\displaystyle\int\limits_{\Omega }|u\left( \tau \right) |^{p-2}u\left( \tau
\right) v_{t}\left( \tau \right) \,dx\,d\tau .%
\end{array}%
\quad  \label{schauder}
\end{equation}%
Using H\"{o}lder's inequality, we can control the last term in the right
hand side of the inequality (\ref{schauder}) as follows:
\begin{equation*}
\displaystyle\int\limits_{0}^{t}\displaystyle\int\limits_{\Omega }|u\left(
\tau \right) |^{p-2}u\left( \tau \right) v_{t}\left( \tau \right) dxd\tau
\leq \displaystyle\int\limits_{0}^{t}\Vert u\left( \tau \right) \Vert
_{2N/\left( N-2\right) }^{p-1}\Vert v_{t}\left( \tau \right) \Vert _{{2N}/{%
\bigl(3N-Np+2(p-1)\bigl)}}d\tau
\end{equation*}%
Since $\displaystyle p\leq \frac{2N}{N-2}$, we have:
$$\displaystyle\frac{2N}{\bigl(3N-Np+2(p-1)\bigl)}\leq \frac{2N}{N-2}\quad .$$
Thus, by Young's and Sobolev's inequalities, we get for all $\delta >0$
there exists $C(\delta )>0$, such that for all $t\in \left( 0,T\right) $
\begin{equation*}
\displaystyle\int\limits_{0}^{t}\displaystyle\int\limits_{\Omega }|u\left(
\tau \right) |^{p-2}u\left( \tau \right) v_{t}\left( \tau \right) dxd\tau
\leq C(\delta )tR^{2(p-1)}+\delta \displaystyle\int\limits_{0}^{t}\Vert
\nabla v_{t}\left( \tau \right) \Vert _{2}^{2}d\tau .
\end{equation*}%
Inserting the last estimate in the inequality (\ref{schauder}) and choosing $%
\delta $ small enough such that:
\begin{equation*}
\Vert v\Vert _{Y_{T}}^{2}\leq \frac{1}{2}R^{2}+CTR^{2(p-1)}.
\end{equation*}%
Thus, for $T$ sufficiently small, we have $\Vert v\Vert _{Y_{T}}\leq R$.
This shows that $v\in B_{R}\left( X_{T}\right) $.

To prove that $\Phi $ is a contraction, we have to follow the same steps (up
to minor changes) as in \cite{GS08}. We omit the details. Thus the proof of
Theorem \ref{existence} is finished.
\end{proof}
\begin{remark}
Let us say that the hypothesis on $m$, $\max\left( 2, \frac{\bar{q}}{\bar{q}+1-p} \right) \leq m \leq \bar{q}$, is made to pass to the limit
in the nonlinear term, by the same way we have used in \cite[Equation (2.28)]{GS08}.
\end{remark}

\section{Global existence}\label{Global_existence_section}

In this section, we show that, under some restrictions on the initial data, 
the local solution of problem (\ref{ondes}) can be continued in time
and the lifespan of the solution will be $[0,\infty )$. 
\begin{definition} \label{Tmax}
Let $2\leq p\leq \bar{q}$, $\max\left( 2, \frac{\bar{q}}{\bar{q}+1-p} \right) \leq m \leq \bar{q}$,
$u_{0}\in H_{\Gamma_{0}}^{1}(\Omega) $ and $u_{1}\in L^{2}(\Omega) $. We denote by $u$ the solution of (\ref{ondes}).
We define:
$$
T_{max} = \sup\Bigl\{ T > 0 \,,\, u = u(t) \ \mbox{ exists on } \ [0,T]\Bigr\} \ .
$$
Since the solution $u \in Y_T$ (the solution is ``regular enough''), from the definition of the norm given by (\ref{norm}), let us recall that if
$T_{max} < \infty$, then
$$
 \lim_{\underset {t < T_{max}} {t \rightarrow T_{max}}} \Vert \nabla u(t) \Vert_2  + \Vert u_t(t) \Vert_2 = + \infty.
$$
If $T_{max} < \infty$, we say that the solution of (\ref{ondes}) blows up and that $T_{max}$ is the blow up time.\\
If $T_{max} = \infty$, we say that the solution of (\ref{ondes})  is global.
\end{definition}

In order to study the blow up phenomenon or the global existence of the solution of (\ref{ondes}),  
for all $0\leq t <  T_{max}$, we define:
\begin{eqnarray}
I(t) &=&I(u(t))=\left( 1-\displaystyle\int_{0}^{t}g\left( s\right) ds\right)
\left\Vert \nabla u\left( t\right) \right\Vert _{2}^{2}+\left( g\diamond
u\right) \left( t\right) -\Vert u\Vert _{p}^{p},  \label{Energy_I} \\
J(t) &=&J(u(t))=\frac{1}{2}\left( 1-\displaystyle\int_{0}^{t}g\left(
s\right) ds\right) \left\Vert \nabla u\left( t\right) \right\Vert _{2}^{2}+%
\dfrac{1}{2}\left( g\diamond u\right) \left( t\right) -\frac{1}{p}\Vert
u\Vert _{p}^{p}.  \label{Energy_J}
\end{eqnarray}%
Thus the energy functional defined in (\ref{Energy_visco_elastic}) can
be rewritten as
\begin{equation}
E(u(t))=E(t)=J(t)+\frac{1}{2}\Vert u_{t}\Vert _{2}^{2}+\frac{1}{2}\Vert
u_{t}\Vert _{2,\Gamma _{1}}^{2}.  \label{Energy_E}
\end{equation}%
As in \cite{GS06,V99}, we denote by $B$ the best constant in the Poincar\'{e}-Sobolev
embedding $H_{\Gamma_{0}}^{1}(\Omega) \hookrightarrow L^{p}(\Omega)$ defined by:
\begin{equation}\label{sobolev}
B^{-1} = \inf\left\{\Vert \nabla u \Vert_2 : u \in  H_{\Gamma_{0}}^{1}(\Omega), \Vert u\Vert_p = 1 \right\}.
\end{equation}

For $u_{0}~\in~H_{\Gamma _{0}}^{1}(\Omega )\,,\,u_{1}\in L^{2}(\Omega )$, we define:
\begin{equation*}
E(0) =\dfrac{1}{2}\Vert u_{1}\Vert_{2}^{2}+\dfrac{1}{2}\Vert u_{1}\Vert_{2,\Gamma _{1}}^{2}+
\dfrac{1}{2} \left\Vert \nabla u_{0} \right\Vert _{2}^{2} -\dfrac{1}{p}\left\Vert u_{0}\right\Vert _{p}^{p}.%
\end{equation*}

The  first goal is to prove that the above energy $E\left( t\right) $ defined
in (\ref{Energy_visco_elastic}) is a non-increasing function along the
trajectories. More precisely, we have the following result:

\begin{lemma}
\label{Lemma_dissp_energy_visco} Let $2\leq p\leq \bar{q}$, $\max\left( 2, \frac{\bar{q}}{\bar{q}+1-p} \right) \leq m \leq \bar{q}$,
and $u$ be the solution of (\ref{ondes}).
Then, for all $t>0,$ we have
\begin{eqnarray}
\frac{dE\left( t\right) }{dt} &=&\frac{1}{2}\left( g^{\prime }\diamond
u\right) \left( t\right) -\frac{1}{2}g\left( t\right) \left\Vert \nabla
u\left( t\right) \right\Vert _{2}^{2}-\alpha \left\Vert \nabla
u_{t}\right\Vert _{2}^{2}-\int_{\Gamma _{1}}h\left( u_{t}\right) u_{t}d\Gamma
\notag \\
&\leq &\frac{1}{2}\left( g^{\prime }\diamond u\right) \left( t\right)
-\alpha \left\Vert \nabla u_{t}\right\Vert _{2}^{2}-\int_{\Gamma
_{1}}h\left( u_{t}\right) u_{t}d\Gamma ,\qquad \forall t\in \lbrack
0,T_{max}).  \label{dissp_enrgy_visco}
\end{eqnarray}
\end{lemma}

\begin{proof}Multiplying the first equation in (\ref{ondes}) by $u_{t}$,
integrating over $\Omega $, using integration by parts we get:
\begin{equation}
\left.
\begin{array}{l}
\dfrac{d}{dt}\left\{ \dfrac{1}{2}\Vert u_{t}\Vert _{2}^{2}+\dfrac{1}{2}\Vert
u_{t}\Vert _{2,\Gamma _{1}}^{2}+\dfrac{1}{2}\Vert \nabla u\Vert _{2}^{2}-%
\dfrac{1}{p}\left\Vert u\right\Vert _{p}^{p}\right\} \vspace{0.2cm} \\
-\displaystyle\int_{\Omega }\displaystyle\int_{0}^{t}g\left( t-s\right)
\nabla u\left( s\right) \nabla u_{t}\left( t\right) dsdx\vspace{0.2cm} \\
=-\alpha \left\Vert \nabla u_{t}\right\Vert _{2}^{2}-\displaystyle%
\int_{\Gamma _{1}}h\left( u_{t}\right) u_{t}d\Gamma .%
\end{array}%
\right.  \label{d_dt_Energy_2}
\end{equation}%
A simple use of the identity (\ref{Nonlinear_term_viscoelastic}) gives (\ref%
{dissp_enrgy_visco}). This completes the proof of Lemma \ref%
{Lemma_dissp_energy_visco}.
\end{proof}

\begin{lemma}
\label{Stable_set_lemma}Let $2\leq p\leq \bar{q}$, $\max\left( 2, \frac{\bar{q}}{\bar{q}+1-p} \right) \leq m \leq \bar{q}$.
Assume that (\ref%
{hypothesis_g}) and (\ref{hypothesis_g_2}) hold. Then given $u_{0}~\in
~H_{\Gamma _{0}}^{1}(\Omega )\,,\,u_{1}\in L^{2}(\Omega )$ satisfying
\begin{equation}
\left\{
\begin{array}{l}
\beta =\dfrac{B^{p}}{l}\left( \dfrac{2p}{l\left( p-2\right) }E(0)^{\left( p-2\right) /2}\right)<1,\vspace{0.2cm} \\
I\left( u_{0}\right) >0,%
\end{array}%
\right.  \label{beta_condition}
\end{equation}%
we have:
\begin{equation*}
I\left( u\left( t\right) \right) >0,\qquad \forall t\in \lbrack 0,T_{\max }).
\end{equation*}
\end{lemma}
\begin{proof} Since $I\left( u_{0}\right) >0$, then by continuity, there
exists $T^{\ast }<T_{\max }$, such that
\begin{equation*}
I\left( t\right) >0,\qquad \forall t\in \lbrack 0,T^{\ast }]
\end{equation*}%
which implies that for all $t\in \lbrack 0,T^{\ast }],$
\begin{eqnarray}
J\left( t\right) &=&\frac{1}{p}I\left( t\right) +\frac{p-2}{2p}\left\{
\left( 1-\displaystyle\int_{0}^{t}g\left( s\right) ds\right) \left\Vert
\nabla u\left( t\right) \right\Vert _{2}^{2}+\left( g\diamond u\right)
\left( t\right) \right\}  \notag \\
&\geq &\frac{p-2}{2p}\left\{ \left( 1-\displaystyle\int_{0}^{t}g\left(
s\right) ds\right) \left\Vert \nabla u\left( t\right) \right\Vert
_{2}^{2}+\left( g\diamond u\right) \left( t\right) \right\} .
\label{J_inequality}
\end{eqnarray}%
By using (\ref{hypothesis_g}), (\ref{hypothesis_g_2}), (\ref{Energy_E}) and (%
\ref{dissp_enrgy_visco}), we easily get, for all $t\in \lbrack 0,T^{\ast }]$
\begin{eqnarray}
l\left\Vert \nabla u(t)\right\Vert _{2}^{2} &\leq &\frac{2p}{p-2}J\left(
t\right) ,\   \notag \\
&\leq &\frac{2p}{p-2}E\left( t\right) \leq \frac{2p}{p-2}E\left( 0\right) .
\label{E_0_J_inequality}
\end{eqnarray}
From the definition of the constant $B$ in (\ref{sobolev}), we first get:
$$
\forall t\in \lbrack 0,T^{\ast }] \ , \ \left\Vert u(t)\right\Vert _{p}^{p} \leq B^{p}\Vert \nabla u(t)\Vert_{2}^{p} \ .
$$
Since we have:
$$
\forall t\in \lbrack 0,T^{\ast }] \ , \ B^{p}\Vert \nabla u(t)\Vert_{2}^{p}= \frac{B^{p}}{l}\Vert \nabla u(t)\Vert _{2}^{p-2}\left( l\Vert\nabla u(t)\Vert _{2}^{2}\right) ,
$$
by exploiting (\ref{E_0_J_inequality}) and (\ref{beta_condition}), we
obtain, for all $t\in \lbrack 0,T^{\ast }]$:
\begin{eqnarray*}
\left\Vert u(t)\right\Vert _{p}^{p} &\leq &\beta l\left( \Vert \nabla u(t)\Vert
_{2}^{2}\right) \\
&\leq &\beta \left( 1-\int_{0}^{t}g\left( s\right) ds\right) \Vert \nabla
u(t)\Vert _{2}^{2} \\
&<&\left( 1-\int_{0}^{t}g\left( s\right) ds\right) \Vert \nabla u(t)\Vert
_{2}^{2}.
\end{eqnarray*}%
Therefore, by using (\ref{Energy_I}), we conclude that
\begin{equation*}
I\left( t\right) >0,\qquad \forall t\in \lbrack 0,T^{\ast }].
\end{equation*}%
Using the fact that $E$ is decreasing along the trajectory, we get:
\begin{equation*}
\forall \, 0\leq t < T_{max} \,,\, \dfrac{B^{p}}{l}\left( \dfrac{2p}{l\left(p-2\right) }E\left( t\right) \right) ^{\left( p-2\right) /2}\leq \beta <1 \ .
\end{equation*}%
By repeating this procedure, $T^{\ast }$ is extended to $T_{max}.$
\end{proof}
Now, we are able to state the global existence theorem.
\begin{theorem}
\label{Global_existence}Let $2\leq p\leq \bar{q}$, $\max\left( 2, \frac{\bar{q}}{\bar{q}+1-p} \right) \leq m \leq \bar{q}$. 
Assume that (\ref{hypothesis_g}) and (\ref{hypothesis_g_2}) hold. Then given $u_{0}~\in
~H_{\Gamma _{0}}^{1}(\Omega )\,,\,u_{1}\in L^{2}(\Omega )$ satisfying (\ref%
{beta_condition}). Then the solution of (\ref{ondes}) is global and bounded.
\end{theorem}

\begin{proof} To prove Theorem \ref{Global_existence}, using the definition of $T_{max}$, we have just to check that%
\begin{equation*}
\left\Vert \nabla u(t)\right\Vert_{2}^{2}+\left\Vert u_{t}(t)\right\Vert _{2}^{2}
\end{equation*}%
is uniformly bounded in time. To achieve this, we use (\ref{Energy_J}), (%
\ref{Energy_E}), (\ref{dissp_enrgy_visco}) and (\ref{E_0_J_inequality}) to
get%
\begin{eqnarray}
E\left( 0\right) &\geq &E\left( t\right) =J\left( t\right) +\frac{1}{2}\Vert
u_{t}\Vert _{2}^{2}+\frac{1}{2}\Vert u_{t}\Vert _{2,\Gamma _{1}}^{2}  \notag
\\
&\geq &\frac{p-2}{2p}\left\Vert \nabla u(t)\right\Vert _{2}^{2}+\frac{1}{2}%
\left\Vert u_{t}(t)\right\Vert _{2}^{2}.  \label{Inq_E_t_E_0}
\end{eqnarray}%
Therefore,
\begin{equation*}
\left\Vert \nabla u(t)\right\Vert_{2}^{2}+\left\Vert u_{t}(t)\right\Vert _{2}^{2}\leq C E(0)
\end{equation*}%
where $C$ is a positive constant, which depends only on $p.$
\end{proof}

\section{Exponential growth for $\protect\alpha>0$}

\label{Exponential_growth_section}

In this section we will prove that when the initial data are large enough, the energy of the solution of problem
(\ref{ondes}) defined by (\ref{Energy_visco_elastic}) grows exponentially and thus so the $L^p$ norm.

In order to state and prove the exponential growth result, we introduce the following constants: 
\begin{equation}
B_{1}=\frac{B}{l},\quad \alpha _{1}=B_{1}^{-p/(p-2)},\quad E_{1}=\left( \frac{1}{2}-\frac{1}{p}\right) \alpha_{1}^{2},\quad E_{2}=\left( \frac{l}{2}-\frac{1}{p}\right) \alpha _{1}^{2}  \label{constant}
\end{equation}
Let us first mention that $E_{2} < E_{1}$.

The following Lemma will play an essential role in the proof of the exponential growth
result, and it is inspired by the work in \cite{CCLa_2007} where the authors proved
a similar lemma for the wave equation.

First, we define the function
\begin{equation}
\gamma \left( t\right) :=\left( 1-\displaystyle\int_{0}^{t}g\left( s\right)
ds\right) \left\Vert \nabla u\left( t\right) \right\Vert _{2}^{2}+\left(
g\diamond u\right) \left( t\right) .  \label{Gama}
\end{equation}
Let us rewrite the energy functional $E$ defined by  (\ref{Energy_visco_elastic}) as:
\begin{equation}\label{Energy_visco_elastic2}
E(t) = \dfrac{1}{2}\Vert u_{t}\left(t\right) \Vert _{2}^{2}+\dfrac{1}{2}\Vert u_{t}\left( t\right) \Vert_{2,\Gamma _{1}}^{2}+
 \dfrac{1}{2}\gamma \left( t\right) - \dfrac{1}{p}\left\Vert u(t)\right\Vert_{p}^{p} \ .
\end{equation}

\begin{lemma}
\label{Vitilaro_Lemma} Let $2\leq p\leq \bar{q}$, $\max\left( 2, \frac{\bar{q}}{\bar{q}+1-p} \right) \leq m \leq \bar{q}$. 
Let $u$ be the solution of (\ref{ondes}). Assume that
\begin{equation}
E\left( 0\right) <E_{1}\quad \text{ and }\quad \left\Vert \nabla
u_{0}\right\Vert _{2}\geq \alpha _{1}.  \label{Initial_data_assumptions}
\end{equation}%
Then there exists a constant $\alpha _{2}>\alpha _{1}$ such that
\begin{equation}
\left( \gamma \left( t\right) \right) ^{1/2}\geq \alpha _{2},\qquad \forall
t\in \lbrack 0,T_{\max })  \label{result_Vitillaro_1}
\end{equation}%
and
\begin{equation}
\left\Vert u\left( t\right) \right\Vert _{p}\geq B_{1}\alpha _{2},\;\qquad
\forall t\in \lbrack 0,T_{\max }).  \label{result_Vitillaro_2}
\end{equation}
\end{lemma}
\begin{proof} We first note that, by (\ref{Energy_visco_elastic2}), we have:
\begin{eqnarray}
E(t) &\geq &\dfrac{1}{2}\gamma \left( t\right) -\dfrac{1}{p}\left\Vert
u\left( t\right) \right\Vert _{p}^{p}  \notag \\
&\geq &\dfrac{1}{2}\gamma \left( t\right) -\frac{B_{1}^{p}}{p}\left(
l\left\Vert \nabla u\left( t\right) \right\Vert _{2}\right) ^{p}  \notag \\
&\geq &\dfrac{1}{2}\gamma \left( t\right) -\frac{B_{1}^{p}}{p}\left( \gamma
\left( t\right) \right) ^{p/2}  \label{F_gamma_1} \\
&=&\dfrac{1}{2}\alpha ^{2}-\frac{B_{1}^{p}}{p}\alpha ^{p}=F\left( \alpha
\right) ,  \notag
\end{eqnarray}%
where $\alpha =\left( \gamma \left( t\right) \right) ^{1/2}.$ It is easy to verify that $F$ is increasing for $0<\alpha <\alpha _{1},$ decreasing for
$\alpha >\alpha _{1},$ $F\left( \alpha \right) \rightarrow -\infty $ as $%
\alpha \rightarrow +\infty ,$ and
\begin{equation*}
F\left( \alpha _{1}\right) =E_{1},
\end{equation*}%
where $\alpha _{1}$ is given in (\ref{constant}). Therefore, since $%
E(0)<E_{1},$\ there exists $\alpha _{2}>\alpha _{1}$ such that $\,\,F\left(
\alpha _{2}\right) =$ $E\left( 0\right) .$\newline
If we set $\alpha _{0}=\left( \gamma \left( 0\right) \right) ^{1/2},$ then
by (\ref{F_gamma_1}) we have:
\begin{equation*}
F\left( \alpha _{0}\right) \leq E\left( 0\right) =F\left( \alpha _{2}\right)
,
\end{equation*}%
which implies that $\alpha _{0}\geq \alpha _{2}.$\newline
Now to establish (\ref{result_Vitillaro_1}), we suppose by contradiction that:
\begin{equation*}
\left( \gamma \left( t_{0}\right) \right) ^{1/2}<\alpha _{2},
\end{equation*}%
for some $t_{0}>0$ and by the continuity of $\gamma \left( .\right) ,$ we
may choose\ $t_{0}$ such that
\begin{equation*}
\left( \gamma \left( t_{0}\right) \right) ^{1/2}>\alpha _{1}.
\end{equation*}%
Using again  (\ref{F_gamma_1}) leads to:
\begin{equation*}
E\left( t_{0}\right) \geq F\left( \gamma \left( t_{0}\right) ^{1/2}\right)>F(\alpha _{2})=E\left( 0\right) .
\end{equation*}%

But this is impossible since  for all $t>0$, $E(t)\leq E\left( 0\right) $. Hence (\ref{result_Vitillaro_1}) is established.

To prove (\ref{result_Vitillaro_2}), we use (\ref{Energy_visco_elastic2}) to get:
\begin{equation*}
\frac{1}{2}\gamma \left( t\right) \leq E\left( 0\right) +\dfrac{1}{p}%
\left\Vert u\left( t\right) \right\Vert _{p}^{p} \ .
\end{equation*}%
Consequently, using (\ref{result_Vitillaro_1}) leads to:
\begin{eqnarray*}
\dfrac{1}{p}\left\Vert u\left( t\right) \right\Vert _{p}^{p} &\geq &\frac{1}{%
2}\gamma \left( t\right) -E\left( 0\right) \\
&\geq &\frac{1}{2}\alpha _{2}^{2}-E\left( 0\right).
\end{eqnarray*}%
But we have:
$$\frac{1}{2}\alpha _{2}^{2}-E\left( 0\right) = \frac{1}{2}\alpha _{2}^{2}-F\left( \alpha _{2}\right) =\frac{B_{1}^{p}}{p}\alpha_{2}^{p} \ .$$
Therefore (\ref{result_Vitillaro_2}) holds. This finishes the proof of Lemma \ref{Vitilaro_Lemma}.
\end{proof}
The exponential growth result reads as follows:

\begin{theorem}
\label{blow_up_viscoel} Suppose that (\ref{hypothesis_g}), (\ref%
{hypothesis_g_2}) (\ref{Assumption_h}) hold. Assume that
\begin{equation*}
2\leq m\quad \text{and}\quad \max \left( m,2/l\right) <p\leq \overline{p}.
\end{equation*}
Then, the solution of (\ref{ondes}) satisfying%
\begin{equation}
E\left( 0\right) <E_{2}\ \mbox{ and } \ \left\Vert \nabla u_{0}\right\Vert _{2}\geq
\alpha _{1},  \label{Initial_data_visco}
\end{equation}%
grows exponentially in the $L^{p}$ norm.
\end{theorem}

\begin{remark}
\label{Remark_Gerbi_Said} It is obvious that for $g=0$, we have $E_1=E_2$,
and Theorem \ref{blow_up_viscoel} reduces to Theorem 3.1 in \cite{GS08}.
\end{remark}

\begin{remark}
\label{Remark_integral_condition} In Theorem \ref{blow_up_viscoel}, the
condition
\begin{equation*}
\int_0^\infty g(s)ds<\frac{(p/2)-1}{(p/2)-1+(1/2p)}
\end{equation*}
used in \cite%
{Mess_Kafi_2007,Mess_Kafi_2008,MS2010,Mes01,Mess03,SaZh2010,YWL2009} is
unnecessary and our result holds without it.
\end{remark}
\begin{remark} \label{remarkc1}
Let us denote $c_{1}=\left(l-\frac{2}{p}\right)-2E_{2}\left( B_{1}\alpha _{2}\right) ^{-p}.$ Since we have seen that
 $\alpha_{2} > \alpha_{1}$, using the definition of $E_{2}$, we easily get  $c_{1}>0$. This constant will
play an important role in the proof of  Theorem \ref{blow_up_viscoel}
\end{remark}
\begin{proof}[Proof of Theorem \protect\ref{blow_up_viscoel}]
\label{subsection_proof_blow_up}

We implement the so-called Georgiev-Todorova method (see \cite{GT94,Mes01}
and also \cite{MS041}). So, we suppose that the solution exists for all time and we will prove an exponential growth. For this purpose, we set:
\begin{equation}
\mathscr{H}\left( t\right) =E_{2}-E\left( t\right) .  \label{function_H}
\end{equation}%
Of course by (\ref{Initial_data_assumptions}) and (\ref{dissp_enrgy_visco})
and since $E_{2}<E_{1}$, we deduce that $\mathscr{H}$ is a non-decreasing
function. So, by using (\ref{Energy_visco_elastic2}) and, (\ref{function_H})
we get successively:
\begin{equation*}
0 < \mathscr{H}\left( 0\right) \leq \mathscr{H}\left( t\right) 
\leq E_{2}-E\left( t\right) 
\leq E_{1}-\dfrac{1}{2}\gamma(t)+\dfrac{1}{p}\left\Vert u\left( t\right) \right\Vert _{p}^{p}. 
\end{equation*}%
On one hand as $F(\alpha_{1}) = E_{1}$ and $ \forall \ t > 0 \,,\ \gamma(t) \geq \alpha_{2}^{2} > \alpha_{1}^{2}$, we obtain:
$$
E_{1}-\dfrac{1}{2}\gamma(t)< F\left( \alpha _{1}\right) -\frac{1}{2}\alpha _{1}^{2}
$$
On the other hand, since
$$F\left( \alpha _{1}\right) -\frac{1}{2}\alpha _{1}^{2}=-\frac{B_{1}^{p}}{p}\alpha _{1}^{p} \ , $$
we obtain the following inequality:
\begin{equation}
0<\mathscr{H}\left( 0\right) \leq \mathscr{H}\left( t\right) \leq \dfrac{1}{p%
}\left\Vert u\left( t\right) \right\Vert _{p}^{p},\;\qquad \forall t\geq 0.
\label{H_inequality}
\end{equation}%
For $\varepsilon $ small to be chosen later, and inspired by the ideas of
the authors in  \cite{GS08}, we then define the auxiliary function:
\begin{equation}
\mathscr{L}\left( t\right) =\mathscr{H}\left( t\right) +\varepsilon
\int_{\Omega }u_{t}udx+\varepsilon \int_{\Gamma _{1}}u_{t}ud\Gamma +\frac{%
\varepsilon \alpha }{2}\left\Vert \nabla u\right\Vert _{2}^{2}.  \label{defL}
\end{equation}%
Let us remark that $\mathscr{L}$ is a small perturbation of the energy. By
taking the time derivative of (\ref{defL}), using problem (\ref{ondes}), we obtain:
\begin{eqnarray}
\frac{d\mathscr{L}(t)}{dt} &=&\alpha \left\Vert \nabla u_{t}\right\Vert
_{2}^{2}+\int_{\Gamma _{1}}h\left( u_{t}\right) u_{t}d\Gamma +\varepsilon
\left\Vert u_{t}\right\Vert _{2}^{2}-\varepsilon \left\Vert \nabla
u\right\Vert _{2}^{2}  \notag \\
&&+\varepsilon \left\Vert u\right\Vert _{p}^{p}+\varepsilon \left\Vert
u_{t}\right\Vert _{2,\Gamma _{1}}^{2}-\varepsilon \int_{\Gamma _{1}}h\left(
u_{t}\right) u(x,t)d\sigma  \notag \\
&&+\int_{\Omega }\nabla u\left( t\right) \int_{0}^{t}g\left( t-s\right)
\nabla u\left( s\right) dsdx.  \label{derivL2}
\end{eqnarray}%
By making use of (\ref{Assumption_h}) and the following Young's inequality
\begin{equation}
XY\leq \frac{\lambda ^{\mu }X^{\mu }}{\mu }+\frac{\lambda ^{-\nu
}Y^{\nu }}{\nu },  \label{Young_inequality}
\end{equation}%
$X,\,Y\geq 0,\;\lambda >0,\;\mu ,\,\nu \in \mathbb{R^{+}}$ such that $%
1/\mu +1/\nu =1,$ then we get%
\begin{eqnarray}
\int_{\Gamma _{1}}h\left( u_{t}\right) ud\Gamma &\leq &C_{m}\int_{\Gamma
_{1}}\left\vert u_{t}\right\vert ^{m-2}u_{t}ud\Gamma  \notag \\
&\leq &C_{m}\frac{\lambda ^{m}}{m}\left\Vert u\right\Vert _{m,\Gamma
_{1}}^{m}+C_{m}\frac{m-1}{m}\lambda ^{-m/\left( m-1\right) }\left\Vert
u_{t}\right\Vert _{m,\Gamma _{1}}^{m}.  \label{Young_1}
\end{eqnarray}%
Now, the term involving $g$ on the right-hand side of (\ref{derivL2}) can be
written as%
\begin{eqnarray}
\displaystyle\int_{\Omega }\nabla u\left( t,x\right) \displaystyle
&&\hspace*{-1cm}\int_{0}^{t}g\left( t-s\right) \nabla u\left( s,x\right) dsdx\vspace{0.2cm}
=\Vert \nabla u\left( t\right) \Vert_{2}^{2}\left( \displaystyle\int_{0}^{t}g\left( s\right) ds\right) \label{Las_term_estimate} \\
&+&\displaystyle\int_{\Omega }\nabla u\left( t,x\right) \displaystyle\int_{0}^{t}g\left( t-s\right) \left( \nabla
u\left( s,x\right) -\nabla u\left( t,x\right) \right) dsdx \notag .
\end{eqnarray}

On the other hand, by using H\"{o}lder's and Young's inequalities, we infer
that for all $\mu >0,$ we get%
\begin{equation}
\left.
\begin{array}{l}
\displaystyle\int_{\Omega }\nabla u\left( t,x\right) \displaystyle%
\int_{0}^{t}g\left( t-s\right) \left( \nabla u\left( s,x\right) -\nabla
u\left( t,x\right) \right) dsdx\vspace{0.2cm} \\
\leq \displaystyle\int_{0}^{t}g\left( t-s\right) \Vert \nabla u\left(
t\right) \Vert _{2}\Vert \nabla u\left( s\right) -\nabla u\left( t\right)
\Vert _{2}\vspace{0.2cm}ds \\
\leq \mu \left( g\diamond u\right) \left( t\right) +\dfrac{1}{4\mu }\left( %
\displaystyle\int_{0}^{t}g\left( s\right) ds\right) \Vert \nabla u\left(
t\right) \Vert _{2}^{2}.%
\end{array}%
\right.  \label{mu_inequality}
\end{equation}%
Inserting the estimates (\ref{Young_1}) and (\ref{Las_term_estimate}) into (%
\ref{derivL2}), taking into account the inequality (\ref{mu_inequality}) and
making use of (\ref{Assumption_h}), we obtain by choosing $\mu =1/2$ and
multiplying by $l$
\begin{eqnarray}
l\mathscr{L}^{\prime }\left( t\right) &\geq &\alpha l\left\Vert \nabla
u_{t}\right\Vert _{2}^{2}+l\left( c_{m}-C_{m}\varepsilon \frac{m-1}{m}%
\lambda ^{-m/\left( m-1\right) }\right) \left\Vert u_{t}\right\Vert
_{m,\Gamma _{1}}^{m}+\varepsilon l\left\Vert u_{t}\right\Vert _{2}^{2}
\notag \\
&&+\varepsilon l\left\Vert u\right\Vert _{p}^{p}+\varepsilon l\left\Vert
u_{t}\right\Vert _{2,\Gamma _{1}}^{2}-C_{m}\varepsilon l\frac{\lambda ^{m}}{m%
}\left\Vert u\right\Vert _{m,\Gamma _{1}}^{m}  \label{dL_dt_2} \\
&&-\frac{\varepsilon l}{2}\left( g\diamond u\right) \left( t\right)
-\varepsilon l\Vert \nabla u\left( t\right) \Vert _{2}^{2}.  \notag
\end{eqnarray}%
We want now to estimate the term involving $\left\Vert u\right\Vert
_{m,\Gamma _{1}}^{m}$ in (\ref{dL_dt_2}). We proceed as in \cite{GS08}. Then,
we have
\begin{equation*}
\left\Vert u\right\Vert _{m,\Gamma _{1}}\leq C\left\Vert u\right\Vert
_{H^{s}(\Omega )},
\end{equation*}%
which holds for:
\begin{equation*}
m\geq 1\quad \mbox{ and }\quad 0<s<1,\quad s\geq \frac{N}{2}-\frac{N-1}{m}>0 ,
\end{equation*}%
where $C$ here and in the sequel denotes a generic positive constant which
may change from line to line.

Recalling the interpolation and Poincar\'{e}'s inequalities (see \cite{LM68})
\begin{eqnarray*}
\left\Vert u\right\Vert _{H^{s}(\Omega )} &\leq &C\left\Vert u\right\Vert
_{2}^{1-s}\left\Vert \nabla u\right\Vert _{2}^{s} ,\\
&\leq &C\left\Vert u\right\Vert _{p}^{1-s}\left\Vert \nabla u\right\Vert_{2}^{s},
\end{eqnarray*}%
we finally have the following inequality:
\begin{equation}
\left\Vert u\right\Vert _{m,\Gamma _{1}}\leq C\left\Vert u\right\Vert
_{p}^{1-s}\left\Vert \nabla u\right\Vert _{2}^{s}.  \label{u_m_estimate}
\end{equation}%
If $s<2/m$, using again Young's inequality, we get:
\begin{equation}
\left\Vert u\right\Vert _{m,\Gamma _{1}}^{m}\leq C\left[ \left( \left\Vert
u\right\Vert _{p}^{p}\right) ^{\frac{m\left( 1-s\right) \mu }{p}}+\left(
\left\Vert \nabla u\right\Vert _{2}^{2}\right) ^{\frac{ms\theta }{2}}\right]
\label{estiGamma1}
\end{equation}%
for $1/\mu +1/\theta =1.$ Here we choose $\theta =2/ms,$ to get $\mu
=2/\left( 2-ms\right) $. Therefore the previous inequality becomes:
\begin{equation}
\left\Vert u\right\Vert _{m,\Gamma _{1}}^{m}\leq C\left[ \left( \left\Vert
u\right\Vert _{p}^{p}\right) ^{\frac{m\left( 1-s\right) 2}{\left(
2-ms\right) p}}+\left\Vert \nabla u\right\Vert _{2}^{2}\right] .
\label{normGamma1}
\end{equation}%
Now, choosing $s$ such that:
\begin{equation*}
0<s\leq \frac{2\left( p-m\right) }{m\left( p-2\right) },
\end{equation*}%
we get:
\begin{equation}
\frac{2m\left( 1-s\right) }{\left( 2-ms\right) p}\leq 1.  \label{choicesm}
\end{equation}%
Once the inequality (\ref{choicesm}) is satisfied, we use the classical
algebraic inequality:
\begin{equation}
z^{\nu }\leq \left( z+1\right) \leq \left( 1+\frac{1}{\omega }\right) \left(
z+\omega \right) \;,\quad \forall z\geq 0\;,\quad 0<\nu \leq 1\;,\quad
\omega \geq 0,  \label{Algebraic_inequality}
\end{equation}%
to obtain the following estimate:
\begin{eqnarray}
\left( \left\Vert u\right\Vert _{p}^{p}\right) ^{\frac{m\left( 1-s\right) 2}{%
\left( 2-ms\right) p}} &\leq &d\left( \left\Vert u\right\Vert _{p}^{p}+%
\mathscr{H}\left( 0\right) \right)  \notag  \label{normeLp} \\
&\leq &d\left( \left\Vert u\right\Vert _{p}^{p}+\mathscr{H}\left( t\right)
\right) \;,\quad \forall t\geq 0 ,
\end{eqnarray}%
where we have set $d=1+1/\mathscr{H}(0)$. Inserting the estimate (\ref%
{normeLp}) into (\ref{estiGamma1}) we obtain the following important
inequality:
\begin{equation}
\left\Vert u\right\Vert _{m,\Gamma _{1}}^{m}\leq C\left[ \left\Vert
u\right\Vert _{p}^{p}+l\left\Vert \nabla u\right\Vert _{2}^{2}+\mathscr{H}%
\left( t\right) \right] .  \label{L_gam_m-norm}
\end{equation}%
Keeping in mind that  $ l=1-\int_{0}^{\infty }g\left( s\right) ds$, in order to control the
term $\left\Vert \nabla u\right\Vert _{2}^{2}$ in equation (\ref{dL_dt_2}),
we preferely use (as $\mathscr{H}(t)>0$), the following estimate:
\begin{equation*}
\left\Vert u\right\Vert _{m,\Gamma _{1}}^{m}\leq C\left[ \left\Vert
u\right\Vert _{p}^{p}+l\left\Vert \nabla u\right\Vert _{2}^{2}+2\mathscr{H}%
\left( t\right) \right] .
\end{equation*}%
which gives finally:
\begin{eqnarray}
\left\Vert u\right\Vert _{m,\Gamma _{1}}^{m} &\leq &C\left[ 2E_{2}+\left( 1+%
\frac{2}{p}\right) \left\Vert u\right\Vert _{p}^{p}-\left\Vert
u_{t}\right\Vert _{2}^{2}-\left\Vert u_{t}\right\Vert _{2,\Gamma
_{1}}^{2}\right.  \notag \\
&&\left. +\left( l-\Big(1-\int_{0}^{t}g\left( s\right) ds\Big)\right)
\left\Vert \nabla u\right\Vert _{2}^{2}-\left( g\diamond u\right) \left(
t\right) \right] .  \label{umgama1}
\end{eqnarray}%
Since $ 1-\int_{0}^{t}g\left( s\right) ds\geq l$, then we obtain from above
\begin{equation}
\left\Vert u\right\Vert _{m,\Gamma _{1}}^{m}\leq C\left[ 2E_{2}+\left( 1+%
\frac{2}{p}\right) \left\Vert u\right\Vert _{p}^{p}-\left\Vert
u_{t}\right\Vert _{2}^{2}-\left\Vert u_{t}\right\Vert _{2,\Gamma
_{1}}^{2}-\left( g\diamond u\right) \left( t\right) \right] .
\label{boundary_important_estimate}
\end{equation}%
Now, inserting (\ref{boundary_important_estimate}) into (\ref{dL_dt_2}),
then we infer that:
\begin{eqnarray}
l\mathscr{L}^{\prime }\left( t\right) &\geq &\alpha l\left\Vert \nabla
u_{t}\right\Vert _{2}^{2}+l\left( c_{m}-C_{m}\varepsilon \frac{m-1}{m}%
\lambda ^{-m/\left( m-1\right) }\right) \left\Vert u_{t}\right\Vert
_{m,\Gamma _{1}}^{m}  \notag \\
&&+\varepsilon \left( l+lC_{m}\frac{\lambda ^{m}}{m}C\right) \left\Vert
u_{t}\right\Vert _{2}^{2}+\varepsilon \left( l+lC_{m}\frac{\lambda ^{m}}{m}%
C\right) \left\Vert u_{t}\right\Vert _{2,\Gamma _{1}}^{2}  \notag \\
&&+\varepsilon \left\{ l-C_{m}l\frac{\lambda ^{m}}{m}C\left( 1+\frac{2}{p}%
\right) \right\} \left\Vert u\right\Vert _{p}^{p}  \label{dL_dt_4} \\
&&+\varepsilon \left( C_{m}l\frac{\lambda ^{m}}{m}C-\frac{l}{2}\right)
\left( g\diamond u\right) \left( t\right) -\varepsilon l\Vert \nabla u\left(
t\right) \Vert _{2}^{2}-2C_{m}\varepsilon l\frac{\lambda ^{m}}{m}CE_{2} \ .
\notag
\end{eqnarray}%
From (\ref{function_H}), we get
\begin{eqnarray*}
\mathscr{H}\left( t\right) &\leq &E_{2}-\dfrac{1}{2}\left( 1-\displaystyle%
\int_{0}^{t}g\left( s\right) ds\right) \left\Vert \nabla u\left( t\right)
\right\Vert _{2}^{2}\vspace{0.2cm} \\
&&-\dfrac{1}{2}\left( g\diamond u\right) \left( t\right) +\dfrac{1}{p}%
\left\Vert u\left( t\right) \right\Vert _{p}^{p} \\
&\leq &E_{2}-\frac{l}{2}\left\Vert \nabla u\left( t\right) \right\Vert
_{2}^{2}-\dfrac{1}{2}\left( g\diamond u\right) \left( t\right) +\dfrac{1}{p}%
\left\Vert u\left( t\right) \right\Vert _{p}^{p}.
\end{eqnarray*}%
This last inequality gives
\begin{equation}
-l\left\Vert \nabla u\left( t\right) \right\Vert _{2}^{2}\geq 2\left( %
\mathscr{H}\left( t\right) -E_{2}+\dfrac{1}{2}\left( g\diamond u\right)
\left( t\right) -\dfrac{1}{p}\left\Vert u\left( t\right) \right\Vert
_{p}^{p}\right) .  \label{Gradient_inequality}
\end{equation}%
Consequently, (\ref{dL_dt_4}) takes the form%
\begin{eqnarray}
l\mathscr{L}^{\prime }\left( t\right) &\geq &\alpha l\left\Vert \nabla
u_{t}\right\Vert _{2}^{2}+l\left( c_{m}-C_{m}\varepsilon \frac{m-1}{m}%
\lambda ^{-m/\left( m-1\right) }\right) \left\Vert u_{t}\right\Vert
_{m,\Gamma _{1}}^{m}  \notag \\
&&+\varepsilon \left( l+lC_{m}\frac{\lambda ^{m}}{m}C\right) \left\Vert
u_{t}\right\Vert _{2}^{2}+\varepsilon \left( l+lC_{m}\frac{\lambda ^{m}}{m}%
C\right) \left\Vert u_{t}\right\Vert _{2,\Gamma _{1}}^{2}  \notag \\
&&+\varepsilon \left\{ l-\frac{2}{p}-C_{m}l\frac{\lambda ^{m}}{m}C\left( 1+%
\frac{2}{p}\right) \right\} \left\Vert u\right\Vert _{p}^{p}-2\varepsilon
E_{2}  \label{dL_dt_5} \\
&&+\varepsilon \left( C_{m}l\frac{\lambda ^{m}}{m}C-\frac{l}{2}+1\right)
\left( g\diamond u\right) \left( t\right) +2\varepsilon \mathscr{H}\left(
t\right) -2C_{m}\varepsilon l\frac{\lambda ^{m}}{m}CE_{2}.  \notag
\end{eqnarray}%
Now to estimate the terms involving $\left\Vert u\right\Vert _{p}^{p}$ and $E_{2}$ in (\ref{dL_dt_5}), we simply write:
$$
\left( l-\frac{2}{p}\right) \left\Vert u\right\Vert_{p}^{p}-2E_{2} = 
\left( l-\frac{2}{p}\right) \left\Vert u\right\Vert _{p}^{p}-2E_{2} \frac{\Vert u\Vert_{p}^p}
{\Vert u\Vert_{p}^p} \ .$$
Then by using  (\ref{result_Vitillaro_2}), we get:
\begin{equation*}
\left( l-\frac{2}{p}\right) \left\Vert u\right\Vert _{p}^{p}-2E_{2}\geq
c_{1}\left\Vert u\right\Vert _{p}^{p},
\end{equation*}%
where $c_{1} > 0$ is defined in Remark \ref{remarkc1}.
Thus, (\ref{dL_dt_5}) becomes:
\begin{eqnarray}
l\mathscr{L}^{\prime }\left( t\right) &\geq &\alpha l\left\Vert \nabla
u_{t}\right\Vert _{2}^{2}+l\left( c_{m}-C_{m}\varepsilon \frac{m-1}{m}%
\lambda ^{-m/\left( m-1\right) }\right) \left\Vert u_{t}\right\Vert
_{m,\Gamma _{1}}^{m}  \notag \\
&&+\varepsilon \left( l+lC_{m}\frac{\lambda ^{m}}{m}C\right) \left\Vert
u_{t}\right\Vert _{2}^{2}+\varepsilon \left( l+lC_{m}\frac{\lambda ^{m}}{m}%
C\right) \left\Vert u_{t}\right\Vert _{2,\Gamma _{1}}^{2}  \notag \\
&&+\varepsilon \left\{ c_{1}-C_{m}l\frac{\lambda ^{m}}{m}C\left( 1+\frac{2}{p%
}\right) -4C_{m}\varepsilon l\frac{\lambda ^{m}}{m}CE_{2}\left( B_{1}\alpha
_{2}\right) ^{-p}\right\} \left\Vert u\right\Vert _{p}^{p}  \label{dL_dt_6}
\\
&&+\varepsilon \left( C_{m}l\frac{\lambda ^{m}}{m}C-\frac{l}{2}+1\right)
\left( g\diamond u\right) \left( t\right) +2\varepsilon \left( \mathscr{H}%
\left( t\right) +C_{m}l\frac{\lambda ^{m}}{m}CE_{2}\right) .  \notag
\end{eqnarray}%
Notice that since $l < 1$, we first have :
$$ \forall \, \lambda > 0\,,\,  C_{m}l\dfrac{\lambda^{m}}{m}C-\dfrac{l}{2}+1>0 \ .$$
At this point, we pick $\lambda $ small enough such that:
\begin{equation*}
c_{1}-C_{m}l\dfrac{\lambda ^{m}}{m}C\left( 1+\dfrac{2}{p}\right)
-4C_{m}\varepsilon l\dfrac{\lambda ^{m}}{m}CE_{2}\left( B_{1}\alpha
_{2}\right) ^{-p}>0 \ .
\end{equation*}%
Once $\lambda $ is fixed, we may choose $\varepsilon $ small enough such
that
\begin{equation*}
\left\{
\begin{array}{l}
c_{m}-C_{m}\varepsilon \dfrac{m-1}{m}\lambda ^{-m/\left( m-1\right) }>0,%
\vspace{0.2cm} \\
\mathscr{L}\left( 0\right) >0.%
\end{array}%
\right.
\end{equation*}%
Consequently, we end up with the estimate:
\begin{equation}
\mathscr{L}^{\prime }\left( t\right) \geq \eta _{1}\left( \left\Vert
u_{t}\right\Vert _{2}^{2}+\left\Vert u_{t}\right\Vert _{2,\Gamma
_{1}}^{2}+\left\Vert u\right\Vert _{p}^{p}+\mathscr{H}\left( t\right)
+E_{2}\right) ,\quad \forall t\geq 0  \ . \label{First_main_inequality}
\end{equation}%
Next, it is clear that, by Young's and Poincar\'{e}'s inequalities,
we have:
\begin{equation}
\mathscr{L}\left( t\right) \leq \gamma \left[ \mathscr{H}\left( t\right)
+\left\Vert u_{t}\right\Vert _{2}^{2}+\left\Vert u_{t}\right\Vert _{2,\Gamma
_{1}}^{2}+\left\Vert \nabla u\right\Vert _{2}^{2}\right] \,\mbox{ for some }%
\gamma >0.  \label{estiL1}
\end{equation}%
Since $\mathscr{H}(t)>0$, then for all $t\geq 0$, we have:
\begin{equation}
\frac{l}{2}\left\Vert \nabla u\right\Vert _{2}^{2}\leq \frac{1}{p}\left\Vert
u\right\Vert _{p}^{p}+E_{2},\quad  \label{Dradient_main_estimate}
\end{equation}%
Thus, the inequality (\ref{estiL1}) becomes:
\begin{equation}
\mathscr{L}\left( t\right) \leq \zeta \left[ \mathscr{H}(t)+\left\Vert
u_{t}\right\Vert _{2}^{2}+\left\Vert u_{t}\right\Vert _{2,\Gamma
_{1}}^{2}+\left\Vert u\right\Vert _{p}^{p}+E_{2}\right] \,\,%
\mbox{ for some
}\zeta >0.  \label{estiL1bis}
\end{equation}%
From the two inequalities (\ref{First_main_inequality}) and (\ref{estiL1bis}%
), we finally obtain the differential inequality:
\begin{equation}
\frac{d\mathscr{L}\left( t\right) }{dt}\geq \mu \mathscr{L}\left( t\right)
\,\,\mbox{ for some }\mu >0.  \label{diffineq}
\end{equation}%
An integration of the previous differential inequality (\ref{diffineq}) between $0$
and $t$ gives the following estimate for the function $\mathscr{L}$:
\begin{equation}
\mathscr{L}\left( t\right) \geq \mathscr{L}\left( 0\right) e^{\mu t}.
\label{estiL2}
\end{equation}%
On the other hand, from the definition of the function $\mathscr{L}$, from inequality (\ref{H_inequality}) and  for small values of the parameter $\varepsilon$,
it follows that:
\begin{equation}
\mathscr{L}\left( t\right) \leq \frac{1}{p}\left\Vert u\right\Vert _{p}^{p}.
\label{estiLp2}
\end{equation}%
From the two inequalities (\ref{estiL2}) and (\ref{estiLp2}) we conclude the
exponential growth of the solution in the $L^{p}$-norm.
\end{proof}
\section{Blow up in finite time for $\protect\alpha=0$}

\label{blow_up_section}

In this section, we prove that in the absence of the strong damping $-\Delta
u_{t}$, (i.e. $\alpha =0$),  the solution of problem (\ref{ondes}) blows
up in finite time that  is it exists  $0<T^{\ast}<\infty $ such that $\left\Vert u(t)\right\Vert_{p}\rightarrow \infty $ as $t\rightarrow T^{\ast }$.

The blow up result reads as follows:
\begin{theorem}
\label{blow_up_} Suppose that (\ref{hypothesis_g}), (\ref{hypothesis_g_2})
and (\ref{Assumption_h}) hold. Assume that
\begin{equation*}
2<m\quad \text{and}\quad \max \left( m,2/l\right) <p\leq \overline{p}.
\end{equation*}%
Then, the solution of (\ref{ondes}) satisfying%
\begin{equation}
E\left( 0\right) <E_{2},\qquad \left\Vert \nabla u_{0}\right\Vert _{2}\geq
\alpha _{1},
\end{equation}%
blows up in finite time. That is $\left\Vert u\left( t\right) \right\Vert
_{p}\rightarrow \infty $ as $t\rightarrow T^{\ast }$ for some $0<T^{\ast
}<\infty $.
\end{theorem}

\begin{remark}
The requirement $m>2$ in Theorem \ref{blow_up_} is technical but it seems
necessary in our proof. The case $m=2$ cannot be handled with the method we use here. But the 
same result can be 
shown for $m=2$ by using the concavity method. See
\cite{GS082} for more details.
\end{remark}

\begin{proof}[Proof of Theorem \ref{blow_up_}] To prove Theorem \ref{blow_up_}, we
suppose that the solution exists for all time and we reach to a
contradiction. Following the idea introduced in \cite{GT94} and developed in
\cite{MS041} and \cite{V99}, we will define a function $\hat{L}$ which is a
perturbation of the total energy of the system and which will satisfy the
differential inequality
\begin{equation}
\frac{d\hat{L}\left( t\right) }{dt}\geq \xi \hat{L}^{1+\nu }\left( t\right) \ ,
\label{Georgiev_inequality}
\end{equation}%
where $\nu >0.$ Inequality (\ref{Georgiev_inequality}) leads to a blow up of the solution in finite time 
$T^{\ast }\geq \hat{L}\left( 0\right) ^{-\nu }\xi ^{-1}\nu ^{-1}$, provided
that $\hat{L}\left( 0\right) >0.$ 

To do so, we define the functional $\hat{L}$ as follows:
\begin{equation}
\hat{L}\left( t\right) =\mathscr{H}^{1-\sigma }(t)+\epsilon \int_{\Omega
}u_{t}udx+\epsilon \int_{\Gamma _{1}}u_{t}ud\Gamma ,  \label{L_hat}
\end{equation}%
where the functional $\mathscr{H}$ is defined in (\ref{function_H}), $\sigma
$ is satisfying%
\begin{equation}
0<\sigma \leq \min \left( \frac{p-m}{p\left( m-1\right) },\frac{p-2}{2p},%
\frac{m-2}{2m},\hat{\sigma}\right) ,  \label{Segma}
\end{equation}%
where $\hat{\sigma}$ is defined later in (\ref{sigma_hat}) and $\epsilon $ is a
small positive constant to be chosen later. Taking the time derivative of $%
\hat{L}(t)$ and following the same steps as in the proof of Theorem \ref%
{blow_up_viscoel}, we get (instead of inequality (\ref{dL_dt_2})), for all $\lambda > 0$, 
\begin{eqnarray}
l\hat{L}^{\prime }\left( t\right) &\geq &lc_{m}\left( 1-\sigma \right) %
\mathscr{H}^{-\sigma }\left( t\right) \left\Vert u_{t}\right\Vert _{m,\Gamma
_{1}}^{m}-C_{m}\epsilon \frac{m-1}{m}\lambda ^{-m/\left( m-1\right)
}\left\Vert u_{t}\right\Vert _{m,\Gamma _{1}}^{m}+\epsilon l\left\Vert
u_{t}\right\Vert _{2}^{2}  \notag \\
&&+\epsilon l\left\Vert u\right\Vert _{p}^{p}+\epsilon l\left\Vert
u_{t}\right\Vert _{2,\Gamma _{1}}^{2}-C_{m}\epsilon l\frac{\lambda ^{m}}{m}%
\left\Vert u\right\Vert _{m,\Gamma _{1}}^{m}  \label{dL_hat_dt_1} \\
&&-\frac{\epsilon l}{2}\left( g\diamond u\right) \left( t\right) -\epsilon
l\Vert \nabla u\left( t\right) \Vert _{2}^{2}.  \notag
\end{eqnarray}%
Next, for large positive $M$, we select $\lambda ^{-m/\left( m-1\right) }=M%
\mathscr{H}^{-\sigma }\left( t\right) $. Then the estimate (\ref{dL_hat_dt_1}%
) takes the form:
\begin{eqnarray}
l\hat{L}^{\prime }\left( t\right) &\geq &\left( lc_{m}\left( 1-\sigma
\right) -MC_{m}\epsilon \frac{m-1}{m}\right) \mathscr{H}^{-\sigma }\left(
t\right) \left\Vert u_{t}\right\Vert _{m,\Gamma _{1}}^{m}+\epsilon
l\left\Vert u_{t}\right\Vert _{2}^{2}  \notag \\
&&+\epsilon l\left\Vert u\right\Vert _{p}^{p}+\epsilon l\left\Vert
u_{t}\right\Vert _{2,\Gamma _{1}}^{2}-C_{m}\epsilon l\frac{M^{-\left(
m-1\right) }}{m}H^{\sigma \left( m-1\right) }\left\Vert u\right\Vert
_{m,\Gamma _{1}}^{m}  \label{dL_hat_dt_2} \\
&&-\frac{\epsilon l}{2}\left( g\diamond u\right) \left( t\right) -\epsilon
l\Vert \nabla u\left( t\right) \Vert _{2}^{2}.  \notag
\end{eqnarray}%
Exploiting (\ref{H_inequality}) and (\ref{u_m_estimate}), we get:
\begin{equation*}
\mathscr{H}^{\sigma \left( m-1\right) }\left\Vert u\right\Vert _{m,\Gamma
_{1}}^{m}\leq C\left\Vert u\right\Vert _{p}^{\left( 1-s\right) m+\sigma
p\left( m-1\right) }\left\Vert \nabla u\right\Vert _{2}^{sm}  \ .
\end{equation*}%
Thus, as in section \ref{Exponential_growth_section},
we have
\begin{equation*}
\left\Vert u\right\Vert _{p}^{\left( 1-s\right) m+\sigma p\left( m-1\right)
}\left\Vert \nabla u\right\Vert _{2}^{sm}\leq C\left[ \left( \left\Vert
u\right\Vert _{p}^{p}\right) ^{\left( \frac{m\left( 1-s\right) }{p}+\sigma
\left( m-1\right) \right) \mu }+\left( \left\Vert \nabla u\right\Vert
_{2}^{2}\right) ^{\frac{ms\theta }{2}}\right] .
\end{equation*}%
Choosing $\mu ,\,\theta ,$ and $s$ exactly as in section \ref%
{Exponential_growth_section} (with strict inequalities), we  choose $\sigma$ that verifies:
\begin{equation}
\sigma \leq \frac{2-ms}{2\left( m-1\right) }\left( 1-\frac{2m\left(
1-s\right) }{\left( 2-ms\right) p}\right) =\hat{\sigma}.  \label{sigma_hat}
\end{equation}
The hypotheses on $m$ and $p$ ensure  to have $0 < \sigma < 1$.

Consequently, we get from above:
\begin{equation}
\mathscr{H}^{\sigma \left( m-1\right) }\left\Vert u\right\Vert _{m,\Gamma
_{1}}^{m}\leq C\left[ \left( \left\Vert u\right\Vert _{p}^{p}\right)
^{\left( \frac{m\left( 1-s\right) }{p}+\sigma \left( m-1\right) \right) \mu
}+\left\Vert \nabla u\right\Vert _{2}^{2}\right] .
\label{H_sigma_inequality}
\end{equation}%
Since,
\begin{equation*}
\left( \frac{m\left( 1-s\right) }{p}+\sigma \left( m-1\right) \right) \frac{2%
}{2-ms}\leq 1,
\end{equation*}%
applying the algebraic inequality (\ref{Algebraic_inequality}), we get:
\begin{eqnarray}
\left( \left\Vert u\right\Vert _{p}^{p}\right) ^{\left( \frac{m\left(
1-s\right) }{p}+\sigma \left( m-1\right) \right) \frac{2}{2-ms}} &\leq
&d\left( \left\Vert u\right\Vert _{p}^{p}+\mathscr{H}\left( 0\right) \right)
\notag \\
&\leq &d\left( \left\Vert u\right\Vert _{p}^{p}+\mathscr{H}\left( t\right)
\right) \;,\quad \forall t\geq 0  \ .\label{u_p_inequality_2}
\end{eqnarray}%
Thus, (\ref{u_p_inequality_2}) together with (\ref{H_sigma_inequality})
leads to (see (\ref{boundary_important_estimate})):
\begin{eqnarray}
\mathscr{H}^{\sigma \left( m-1\right) }\left\Vert u\right\Vert _{m,\Gamma
_{1}}^{m} &\leq &Cd\left[ \left\Vert u\right\Vert _{p}^{p}+l\left\Vert
\nabla u\right\Vert _{2}^{2}+\mathscr{H}\left( t\right) \right]  \notag \\
&\leq &Cd\left[ 2E_{2}+\left( 1+\frac{2}{p}\right) \left\Vert u\right\Vert
_{p}^{p}-\left\Vert u_{t}\right\Vert _{2}^{2}-\left\Vert u_{t}\right\Vert
_{2,\Gamma _{1}}^{2}-\left( g\diamond u\right) \left( t\right) \right] .
\label{H_sigma_m_inequality}
\end{eqnarray}%
Inserting (\ref{H_sigma_m_inequality}) into (\ref{dL_hat_dt_2}) and using (%
\ref{Gradient_inequality}), we obtain:
\begin{eqnarray}
l\hat{L}^{\prime}(t)&\geq &\left( lc_{m}\left( 1-\sigma
\right) -MC_{m}\epsilon \frac{m-1}{m}\right) \mathscr{H}^{-\sigma }\left(
t\right) \left\Vert u_{t}\right\Vert _{m,\Gamma _{1}}^{m}  \notag \\
&+~\epsilon& l\left( 1+C_{m}\epsilon \frac{M^{-\left( m-1\right) }}{m}%
Cd\right) \left\{ \left\Vert u_{t}\right\Vert _{2}^{2}+\left\Vert
u_{t}\right\Vert _{2,\Gamma _{1}}^{2}\right\} +2\epsilon \mathscr{H}\left(
t\right) -2\epsilon E_{2}  \label{dL_hat_dt_3} \\
&+~\epsilon&\left\{ l-\frac{2}{p}-C_{m}l\frac{M^{-\left( m-1\right) }}{m}%
Cd\left( 1+\frac{2}{p}\right) \right\} \left\Vert u\right\Vert
_{p}^{p}+C_{m}\epsilon l\frac{M^{-\left( m-1\right) }}{m}Cd\left( g\diamond
u\right) \left( t\right)   \notag\\
&-2~\epsilon& C_{m} l\frac{M^{-\left( m-1\right) }}{m}CdE_{2}+\epsilon \left(
1-\frac{l}{2}\right) \left( g\diamond u\right) \left( t\right) .  \notag
\end{eqnarray}%
Writing again 
$E_{2} = E_{2} {\Vert u\Vert_{p}^p}/{\Vert u\Vert_{p}^p}$
and using again  (\ref{result_Vitillaro_2}), we deduce that:
\begin{eqnarray*}
l\hat{L}^{\prime }\left( t\right)& \geq&  \left( lc_{m}\left( 1-\sigma
\right) -MC_{m}\epsilon \frac{m-1}{m}\right) \mathscr{H}^{-\sigma }\left(
t\right) \left\Vert u_{t}\right\Vert _{m,\Gamma _{1}}^{m} \\
&+~\epsilon& l\left( 1+C_{m}\epsilon \frac{M^{-\left( m-1\right) }}{m}%
Cd\right) \left\{ \left\Vert u_{t}\right\Vert _{2}^{2}+\left\Vert
u_{t}\right\Vert _{2,\Gamma _{1}}^{2}\right\} +2\epsilon \mathscr{H}\left(
t\right) +2C_{m}\epsilon l\frac{M^{-\left( m-1\right) }}{m}CdE_{2} \\
&+~\epsilon& \left\{ l-\frac{2}{p}-2E_{2}\left( B_{1}\alpha _{2}\right)
^{-p}-C_{m}l\frac{M^{-\left( m-1\right) }}{m}Cd\left( 1+\frac{2}{p}\right)
-4C_{m}l\frac{M^{-\left( m-1\right) }}{m}CdE_{2}\left( B_{1}\alpha
_{2}\right) ^{-p}\right\} \left\Vert u\right\Vert _{p}^{p} \\
&+~\epsilon&C_{m} l\frac{M^{-\left( m-1\right) }}{m}Cd\left( g\diamond
u\right) \left( t\right) .
\end{eqnarray*}
Thus, using the definition of $c_{1}$ in Remark \ref{remarkc1}, we get:
\begin{eqnarray*}
l\hat{L}^{\prime }\left( t\right) &\geq &\left( lc_{m}\left( 1-\sigma
\right) -MC_{m}\epsilon \frac{m-1}{m}\right) \mathscr{H}^{-\sigma }\left(
t\right) \left\Vert u_{t}\right\Vert _{m,\Gamma _{1}}^{m} \\
&+~\epsilon& l\left( 1+C_{m}\epsilon \frac{M^{-\left( m-1\right) }}{m}%
Cd\right) \left\{ \left\Vert u_{t}\right\Vert _{2}^{2}+\left\Vert
u_{t}\right\Vert _{2,\Gamma _{1}}^{2}\right\} +2\epsilon \mathscr{H}\left(
t\right) +2C_{m}\epsilon l\frac{M^{-\left( m-1\right) }}{m}CdE_{2} \\
&+~\epsilon& \left\{c_{1}-C_{m}l\frac{M^{-\left( m-1\right) }}{m}Cd\left( 1+\frac{2}{p}\right)
-4C_{m}l\frac{M^{-\left( m-1\right) }}{m}CdE_{2}\left( B_{1}\alpha
_{2}\right) ^{-p}\right\} \left\Vert u\right\Vert _{p}^{p} \\
&+~\epsilon&C_{m}l\frac{M^{-\left( m-1\right) }}{m}Cd\left( g\diamond
u\right) \left( t\right) .
\end{eqnarray*}%
Since $c_{1} > 0$, we choose $M$ large enough such that:
\begin{equation*}
c_{1}-C_{m}l\frac{M^{-\left( m-1\right) }}{m}Cd\left( 1+\frac{2}{p}\right)
-4C_{m}l\frac{M^{-\left( m-1\right) }}{m}CdE_{2}\left( B_{1}\alpha
_{2}\right) ^{-p}>0.
\end{equation*}%
Once $M$ is fixed, we pick $\epsilon $ small enough such that
\begin{equation*}
lc_{m}\left( 1-\sigma \right) -MC_{m}\epsilon \frac{m-1}{m}>0
\end{equation*}%
and $\hat{L}\left( 0\right) >0$. This leads to
\begin{equation}
\hat{L}^{\prime }\left( t\right) \geq \hat{\eta}\left( \left\Vert
u_{t}\right\Vert _{2}^{2}+\left\Vert u_{t}\right\Vert _{2,\Gamma _{1}}^{2}+%
\mathscr{H}\left( t\right) +\left\Vert u\right\Vert _{p}^{p}+E_{2}\right)
\label{L_hat_prime}
\end{equation}%
for some $\hat{\eta}>0.$

On the other hand, it is clear from the definition (\ref{L_hat}), we have:
\begin{equation}
\hat{L}^{\frac{1}{1-\sigma }}\left( t\right) \leq C\left( \epsilon ,\sigma
\right) \left[ \mathscr{H}\left( t\right) +\left( \int_{\Omega
}u_{t}\,udx\right) ^{\frac{1}{1-\sigma }}+\left( \int_{\Gamma
_{1}}u_{t}ud\Gamma \right) ^{\frac{1}{1-\sigma }}\right] \ .
\label{L_second_estimate}
\end{equation}%
By the Cauchy-Schwarz inequality and H\"{o}lder's inequality, we have:
\begin{eqnarray*}
\int_{\Omega }u_{t}udx &\leq &\left( \int_{\Omega }u_{t}^{2}dx\right) ^{%
\frac{1}{2}}\left( \int_{\Omega }u^{2}dx\right) ^{\frac{1}{2}} \\
&\leq &C\left( \int_{\Omega }u_{t}^{2}dx\right) ^{\frac{1}{2}}\left(
\int_{\Omega }\left\vert u\right\vert ^{p}dx\right) ^{\frac{1}{p}},
\end{eqnarray*}%
where $C$ is the positive constant which comes from the embedding $%
L^{p}\left( \Omega \right) \hookrightarrow L^{2}\left( \Omega \right) $.
This inequality implies that there exists a positive constant $C_{1}>0$ such
that:
\begin{equation*}
\left( \int_{\Omega }u_{t}udx\right) ^{\frac{1}{1-\sigma }}\leq C_{1}\left[
\left( \int_{\Omega }\left\vert u\right\vert ^{p}dx\right) ^{\frac{1}{%
(1-\sigma )p}}\left( \int_{\Omega }u_{t}^{2}dx\right) ^{\frac{1}{2(1-\sigma )%
}}\right] .
\end{equation*}%
Applying Young's inequality to the right hand-side of the preceding
inequality, there exists a positive constant also denoted $C>0$ such that:
\begin{equation}
\left( \int_{\Omega }u_{t}udx\right) ^{\frac{1}{1-\sigma }}\leq C\left[
\left( \int_{\Omega }\left\vert u\right\vert ^{p}dx\right) ^{\frac{\tau }{%
(1-\sigma )p}}+\left( \int_{\Omega }u_{t}^{2}dx\right) ^{\frac{\theta }{%
2(1-\sigma )}}\right] ,  \label{Young_main}
\end{equation}%
for $1/\tau +1/\theta =1$. We take $\theta =2(1-\sigma )$, hence $\tau
=2\left( 1-\sigma \right) /\left( 1-2\sigma \right) $, to get:
\begin{equation*}
\left( \int_{\Omega }u_{t}udx\right) ^{\frac{1}{1-\sigma }}\leq C\left[
\left( \int_{\Omega }\left\vert u\right\vert ^{p}dx\right) ^{\frac{2}{%
(1-2\sigma )p}}+\int_{\Omega }u_{t}^{2}dx\right] \ .
\end{equation*}%
Using the algebraic inequality (\ref%
{Algebraic_inequality}) with $z=\left\Vert u\right\Vert _{p}^{p}$, $%
\displaystyle d=1+1/\mathscr{H}(0)$, $\omega =\mathscr{H}(0)$ and $\nu =%
\displaystyle\frac{2}{p\left( 1-2\sigma \right) }$ (the condition (\ref%
{Segma}) on $\sigma $ ensuring that $0<\nu \leq 1$) we get:
\begin{equation*} 
z^{\nu }\leq d\left( z+\mathscr{H}(0)\right) \leq d\left( z+\mathscr{H}%
(t)\right) .
\end{equation*}%
Therefore, there exists a positive constant denoted $C_{2}$ such that \ for
all $t\geq 0$,%
\begin{equation}
\left( \int_{\Omega }u_{t}udx\right) ^{\frac{1}{1-\sigma }}\vspace{0.2cm}%
\leq C_{2}\left[ \mathscr{H}\left( t\right) +\left\Vert u\left( t\right)
\right\Vert _{p}^{p}+\left\Vert u_{t}\left( t\right) \right\Vert _{2}^{2}%
\right] .  \label{estimate_first_term}
\end{equation}%
Following the same method as above, we can show that there exists $C_{3}>0$
such that:
\begin{equation*}
\left( \int_{\Gamma _{1}}u_{t}ud\Gamma \right) ^{\frac{1}{1-\sigma }}\leq
C_{3}\left[ \mathscr{H}\left( t\right) +\left\Vert u\left( t\right)
\right\Vert _{m,\Gamma _{1}}^{m}+\left\Vert u_{t}\left( t\right) \right\Vert
_{2,\Gamma _{1}}^{2}\right] .
\end{equation*}%
Applying the inequality (\ref{L_gam_m-norm}), we get:
\begin{equation*}
\left( \int_{\Gamma _{1}}u_{t}ud\Gamma \right) ^{\frac{1}{1-\sigma }}\leq
C_{4}\left[ \mathscr{H}\left( t\right) +\left\Vert u\left( t\right)
\right\Vert _{p}^{p}+l\left\Vert \nabla u\left( t\right) \right\Vert
_{2}^{2}+\left\Vert u_{t}\left( t\right) \right\Vert _{2,\Gamma _{1}}^{2}%
\right] .
\end{equation*}%
Furthermore, inequality (\ref{Dradient_main_estimate}) leads to:
\begin{equation}
\left( \int_{\Gamma _{1}}u_{t}ud\Gamma \right) ^{\frac{1}{1-\sigma }}\leq
C_{5}\left[ \mathscr{H}\left( t\right) +\left\Vert u\left( t\right)
\right\Vert _{p}^{p}+\left\Vert u_{t}\left( t\right) \right\Vert _{2,\Gamma
_{1}}^{2}+E_{2}\right] .  \label{estimate_second_term}
\end{equation}%
Collecting (\ref{L_second_estimate}), (\ref{estimate_first_term}) and (\ref%
{estimate_second_term}), we obtain:
\begin{equation}
\hat{L}^{\frac{1}{1-\sigma }}\left( t\right) \leq \hat{\eta}_{1}\left\{
\left\Vert u_{t}\left( t\right) \right\Vert _{2}^{2}+\left\Vert
u_{t}\right\Vert _{2,\Gamma _{1}}^{2}+\mathscr{H}(t)+\left\Vert u\left(
t\right) \right\Vert _{p}^{p}+E_{2}\right\} ,\qquad \forall t\geq 0,
\label{L_hat_last}
\end{equation}%
for some $\hat{\eta}_{1}>0.$

Combining (\ref{L_hat_prime}) and (\ref{L_hat_last}), then, there exists a
positive constant $\xi >0$, as small as $\epsilon $, such that for all $%
t\geq 0$,
\begin{equation}
\hat{L}^{\prime }(t)\geq \xi \hat{L}^{\frac{1}{1-\sigma }}(t).
\label{L_1-sigma_2}
\end{equation}%
Thus, inequality (\ref{Georgiev_inequality}) holds. Therefore, $\hat{L}(t)$
blows up in a finite time $T^{\ast }$.

On the other hand, from the definition of the function $\hat{L}(t)$ and using  inequality (\ref{H_inequality}), for small values of the parameter $\varepsilon$,
it follows that:
\begin{equation}
\hat{L}(t)\leq \kappa \left(\left\Vert u\left( t\right) \right\Vert_{p}^{p}\right)^{1-\sigma} \ ,
\label{inq_L_norm}
\end{equation}%
where $\kappa $ is a positive constant. Consequently, from the inequality (%
\ref{inq_L_norm}) we conclude that the norm $\left\Vert u\left( t\right)
\right\Vert _{p}$ of the solution $u$, blows up in the finite time $T^{\ast
} $, which implies the desired result. This completes the proof of Theorem %
\ref{blow_up_}. 
\end{proof}

\end{document}